\documentclass[12pt]{article}
\usepackage{amsmath}
\usepackage{amssymb}
\usepackage{accents}

\usepackage{mathtools}
\usepackage{amsthm}
\usepackage[usenames]{color}

\usepackage{graphicx}

\newcommand{\1}[1]{{\mathbf 1}{\{#1\}}}

\newcommand{\eps}{\varepsilon}
\newcommand{\Z}{{\mathbb Z}}

\newcommand{\B}{{\mathsf B}}
\newcommand{\Psf}{{\mathsf P}}

\newcommand{\R}{{\mathbb R}}

\let\phi=\varphi

\newcommand{\E}{{\mathbb E}}

\newcommand{\8}{{\infty}}

\newcommand{\nn}{{\nonumber}}

\newcommand{\eqlaw}{\stackrel{\text{\tiny law}}{=}}

\newcommand{\IP}{{\mathbb P}}

\DeclareMathSymbol{\widehatsym}{\mathord}{largesymbols}{"62}

\newcommand{\hW}{\widehat{W}}

\newcommand{\ka}{\kappa}
\newcommand{\Ka}{\mathcal K}

\allowdisplaybreaks

\newtheorem{theo}{Theorem}[section]
\newtheorem{lem}[theo]{Lemma}

\newtheorem{prop}[theo]{Proposition}
\newtheorem{cor}[theo]{Corollary}
\newtheorem{rem}[theo]{Remark}

\newcommand{\Bes}{\rm BES}
\newcommand{\dis}{{\displaystyle}}

\title{Rate of escape of conditioned Brownian motion}
\author{Orph\'ee Collin$^{~1}$, Francis Comets$^{~2}$}

\begin{document}

\maketitle

{\footnotesize 
\noindent $^{~1}~$DMA, \'Ecole normale sup\'erieure, Universit\'e PSL, CNRS, 75005 Paris, France\\
\noindent e-mail: \texttt{orphee.collin@normalesup.org}\\
\noindent $^{~2}~$Universit\'e de Paris and LPSM,
Math\'ematiques, 
 case 7012, F--75205 Paris
Cedex 13, France
\\
\noindent e-mail:
\texttt{comets@lpsm.paris}
}
\begin{abstract} 
%%%%%%%   guillemets anglais USA:     ci-dessous
%``aa''
%``a"
We
 study the norm of the two-dimensional Brownian motion conditioned to stay outside the unit disk at all times. 
By conditioning the process is changed from barely recurrent to slightly transient.
%% We obtain sharp results on the rate of escape to infinity  of the process of future minima. 
We obtain sharp results on the rate of escape to infinity  of the process of future minima: \\
\phantom{} \qquad (i) we find an integral test on the function $g$ so that the future minima process drops beyond the barrier $\exp \{ \ln t \times g(\ln \ln t)\}$ at arbitrary large times; \\
\phantom{} \qquad (ii) we show that the     future minima process exceeds $K \sqrt{ t \times \ln \ln \ln t}$ at arbitrary large times with probability 0 [resp., 1] if $K$ is larger [resp., smaller] than some 
positive constant.\\
For this, we introduce a renewal structure attached to record times and values. Additional results  are given for the long time behavior of the norm. 
\\[.3cm]\textbf{Keywords:} Brownian motion, Bessel process, conditioning,
transience, Wiener moustache, regeneration, upper-class and lower-class, random difference equation, autoregressive process
\\[.3cm]\textbf{AMS 2020 subject classifications:}
60K35, 60J60, 60J65, 60G17
\end{abstract}

%{\bf To do list:}
%\begin{enumerate}
%%\item preuve grandes valeurs ptag
%%\item completer les estimees techniques; cf. [REFE] , ??,\ldots 
%%\item abstract
%\item                                                                           figure 1
%\item Quelques ?? \`a trancher
%%\item Notation BES$^2$ ou Bes$^2$
%\end{enumerate}
% 
 
\tableofcontents 

\section{Introduction}

This paper is devoted to the planar Brownian motion conditioned to stay outside the unit ball $\B(0,1)$ at all times. Besides its own appeal from its fundamental character, this process has attracted a keen interest as being the elementary brick of the two-dimensional Brownian random interlacement recently introduced in \cite{CP20}.
By rotational symmetry,  the norm  $R$ of the conditioned Brownian motion  itself follows a %one-dimensional diffusion in $[1,\8)$,  
stochastic differential equation  in $[1,\8)$,  
\begin{equation} 
 d R(t) = \Big(\frac{1}{R(t)\ln R(t)} + \frac{1}{2R(t)}\Big)dt 
 + dB(t)
   \label{df_Rt} 
\end{equation}
with $B$ a standard Brownian motion in $\R$, and  we can -- and we will -- restrict the study of the conditioned process to that of $R$ itself since the angle obeys a diffusion subordinated to it.  The two-dimensional Brownian motion is critically recurrent, but conditioning it outside  the unit ball turns it into %slightly 
(delicately) transient. A natural question is the rate at which  $R(t)$ tends to $\8$ as $t \to \8$, this is the object of the present paper. A measure of the reluctance of $R$ to tend to infinity is given by the {\it future minima process}  
\begin{equation} \label{def:futuremin}
M(t)= \inf \{ R(s); s \geq t \}
\end{equation}
which is non-decreasing to $\8$ a.s. 
The corresponding model in the discrete case, the two-dimensional simple random walk conditioned to avoid the origin at all times, has motivated many recent papers. Estimates on the future minimum distance to the origin have been obtained in \cite{PRU}, we will use them as benchmarks. It is also shown that two independent conditioned walkers meet infinitely often although they are transient. The range of the walk, i.e. the set of visited sites,  
is studied in  \cite{GPV19}:  if a finite $A \subset \Z^2 \setminus \{0\}$ is ``big enough and well distributed in space'', then the proportion of visited sites is approximately uniformly distributed on $[0,1]$.
In \cite{P19} the explicit formula for the Green function  is obtained, and a survey is given in Chapter 4 of \cite{P2srw}.

For dimensions $d \geq 3$, the random interlacement model has been introduced in \cite{Szn10} to describe the local picture of the visited set by a random walk at large times on a large $d$-dimensional torus, and similarly  in  \cite{Sz13},  the Brownian random interlacement  to describe the Wiener sausage around the Brownian motion on a 
$d$-dimensional torus.
For dimension $d =2$, the random interlacement model is the local limit of the visited set by the random walk {\it around a point which has not been visited so far}  \cite{CPV16}, and  analogously,  the Brownian random interlacement is the local limit of the Wiener sausage on the two-dimensional torus {\it around a point which is outside  the sausage}
  \cite{CP20}. Formally, the  two-dimensional Brownian random interlacement is defined as a Poisson process of bi-infinite paths, which are rescaled instances of the so-called ``Wiener moustache". The Wiener moustache is obtained by 
  gluing two instances (for positive  and  negative times, see Figure 1 in  \cite{CP20}) of planar Brownian motion conditioned to stay outside the unit ball, which are independent except that they share  the same starting point (see Lemma 3.9 in \cite{CP20}). Hence, the process we consider in this paper is the building brick of Brownian random interlacement in the plane. We also recall that the complement of the sausage around the interlacement has an interesting phase transition, changing from a.s. unbounded to a.s. bounded as the Poisson intensity is increased,  see Th. 2.13 in  \cite{CP20} and \cite{CP17} for the discrete case.

With a slight abuse of terminology, we say $f(t) \leq g(t) \; i.o.$ (infinitely often) if the set $\{t \geq 0: f(t)\leq g(t)\}$ is unbounded, and 
 $f(t) \leq g(t) \; ev.$ (eventually) if the set $\{t \geq 0: f(t)\leq g(t)\}$ is a neighborhood of $\8$ in $\R_+$.
\medskip

We now give a %partial 
short overview of some of our results on the rate of escape of $R$ to infinity. They are consequences of the results in section \ref{sec:res}.
%%%%%%%
\begin{theo} \label{cor:Kov}
%For $g$ as in Theorem \ref{th:critint},
  For $g: \R_+ \to \R_+$  non-increasing such that $(\ln t) g(\ln \ln  t)$
   is non-decreasing,  
%   \\  \phantom{} \quad $\IP(  M(t) \leq e^{(\ln t) g(\ln \ln t)}\; {i.o.})= 0 {\rm\ or\ } 1 $
%according to $\int^\8 g(u) du < \8$ or $= \8$.
\begin{quote}
 $\IP\Big(  M(t) \leq e^{(\ln t) g(\ln \ln t)}\; {i.o.}\Big)= 
\left\{
\begin{array}{c}
  0   \\
  1  
\end{array}
\right.
 $
\quad according to \quad $\int^\8 g(u) du 
\left\{
\begin{array}{c}
 < \8  \\
 = \8
\end{array}
\right.
$.
\end{quote}
\end{theo}
This result with an integral condition %is reminiscent 
has a flavor of Kolmogorov's test (see, e.g., sect. 4.12 in \cite{ImcK}).
%, but the process $M$ here is not Markov.
%%%%%
\begin{theo} \label{cor:OC}
The limit 
$$
K^*=\limsup_{t \to \8} \frac{M(t)}{\sqrt{t \ln \ln \ln  t}} 
$$
is a.s. constant, and $0<K^*<\8$.
\end{theo}
%%%%%%%
Though we do not know the actual value of $K^*$ we can see that both theorems are much finer than the corresponding Theorem 1.2 of \cite{PRU}. %So are the following ones:
These two theorems together yield  a precise version of the observation from \cite{P19} that the 
pathwise divergence of $R$ 
to infinity occurs in a highly irregular way.
The future minima process has been considered earlier, e.g. \cite{KLL94} and \cite{KLS96} for Bessel processes and for random walks, and \cite{Pardo} for positive self-similar Markov processes.
%Let us recall the similar result for  transient Bessel processes\footnote{ 
%We recall that $\Bes^d$ is the solution of $dX(t)= \frac{d-1}{2X(t)}dt + dB(t)$, it is the norm of the $d$-dimensional Brownian motion for $d$ integer. 
%}, Th. 4.1 in \cite{KLL94}: 
Let us recall the similar result for transient Bessel processes. 
Denote by $\Bes^d$ the $d$-dimensional Bessel process, i.e. the solution of the stochastic differential equation 
\begin{equation} \label{def:Bes}
dX(t)= \frac{d-1}{2X(t)}dt + dB(t)\;,
\end{equation}
 that is the norm of the standard Brownian motion in $\R^d$ when $d$ integer :
then, by Th. 4.1 in \cite{KLL94},
\begin{equation}\label{eq:limsupBES}
  {\rm for} \; d >2, \quad \limsup_{t \to \8} \frac{ \min \{ \Bes^d(s); s \geq t\}}{\sqrt{ 2 t \ln \ln t}}=1.
\end{equation}

An important (and beautiful) finding of our work is a renewal structure in Section \ref{sec:reg} which allows sharp estimates.  To illustrate that, let's mention that we will find a sequence of relevant random variables  $S_n>0$ solving a random difference equation
\begin{equation}\label{eq:ar}
S_n = \alpha_n S_{n-1} + \beta_n \;, \quad n \geq 1,
\end{equation}
 where the  sequence $(\alpha_n,\beta_n)_n$ is i.i.d. with positive coefficients, $\alpha_n < 1$  and $\beta_n$ with logarithmic tails,
 $\IP( \beta_1>t) \sim c/\ln t$ for large $t$. Although autoregressive processes AR(1) of the type \eqref{eq:ar} are usually addressed with exponential or power-law tail for $\beta_n$ \cite{BDMikosch}, the case of logarithmic tail has been also considered,  see
\cite{Kellerer}, \cite{ZGlynn}, \cite{BabillotBougerolElie}, and  also both papers \cite{Alsmeyer} and \cite{Zerner18} for a  recent account.
Interestingly, our model is critical  in the perspective of the Markov  chain $S_n$, in the sense that the actual value of  the constant $c$ is precisely the transition from 
recurrence to transience for the chain.

%It turns out that our model precisely falls in the critical case for the recurrence/transience dychotomy of the Markov 
% chain $S_n$. 

%\medskip
%
%Def: upper/lower class \cite{KLL94}
%Transient Bessel 

\medskip

The paper is organized as follows. We give   the main results in the next section.  The regeneration structure is defined
in Section \ref{sec:reg} , together with the basic estimates, and ending with Remark \ref{rem:RDE} on the above random difference equation. In the next section we prove some results showing that $R$ somewhat behaves at large times like the two-dimensional Bessel process.  In Sections \ref{sec:prth:crintint} and \ref{sec:grandesvaleurs} we prove the two above theorems.

%
%Time-series aspect: very heavy tails. Refe Wintenberger, Mikosch
%%%%%%%%%%%%%%%%%%%%%%%

\section{Main results} 

  We first collect a few properties of the  involved processes. 
  \medskip
  
  We start with some notations. Consider $W$ a two-dimensional standard Brownian motion and denote by $\Psf_x$ the law of $W$ starting at $x$, $\hW$ a  Brownian motion conditioned to stay outside the unit ball,  and denote by $\widehat \Psf_x$ its law starting at $x$, and $R=|\hW|$ its Euclidean norm with $P_r$ the corresponding law ($r=|x|)$. In this paper we are mainly interested in $\IP=P_1$. The construction of the process starting from $R(0)>1$ is standard from taboo process theory, and the one starting from $R(0)=1$ is given in definition 2.2 of \cite{CP20}.\medskip
  
   For a closed subset $B$ of the state space of a process $Y$, we denote the entrance time 
  $\tau(Y;B)=\inf\{ t \geq 0: Y(t)\in B\}$, and write for short $\tau(Y;r)=\tau(Y; \partial \B (0,r) )$ and also 
  $\tau(r)= \tau(R;r)$ when $Y=R$.
  The function $h(x)=\ln |x|$ is harmonic in $ \R^2 \setminus \{0\}$, positive on $ \R^2 \setminus \B(0,1)$ and vanishes on the unit circle. Then, the law $\widehat \Psf_x$ of the planar Brownian motion $W$ 
conditioned outside $\B(0,1)$ is given by Doob's $h$-transform of $\Psf_x$. 
   By definition, for $A \subset {\cal C}(\R^+,  \R)$  which is ${\cal F}_{\tau(r_1)}$-measurable ($1< |x|=r<r_1$) 
\begin{eqnarray}\label{eq:MBconditionne}
%P_r( R \in A) &=& P_x ( |W| \in A \big \vert \tau^{|W|}(r_1) < \tau^{|W|}(1) ) \\ \nn &=& P_x ( |W| \in A, \tau^{|W|}(r_1) < \tau^{|W|}(1) ) \times \frac{\ln r_1}{\ln |x|}
P_r( R \in A) &=& \Psf_x ( |W| \in A \big \vert \tau( W,  r_1)< \tau(W, 1)) \\ \nn &=& \Psf_x ( |W| \in A,  
 \tau( W,  r_1)< \tau(W, 1)    ) \times \frac{\ln r_1}{\ln |x|}
\end{eqnarray}
recalling that 
$ \Psf_x(\tau( W,  r_1)< \tau(W, 1)) = \frac{\ln |x|}{\ln r_1}$ since
$\ln |x|$ is harmonic in $\R^2 \! \setminus \! \{0\}$. 
  
 Another remarkable  property is Remark 3.8 in \cite{CP20} :
For all $x \notin \B(1), \rho>0$, we have 
\[
\widehat \Psf_x \big[ \tau(\hW; \B(y,\rho))<\8\big] \to \frac{1}{2}
 \qquad \text{as }  |y|\to \8 \;.
\]
The scale function for the process $R$ -- that is, the unique (up to affine transformation) real function such that $S(R(t))$ is a local martingale -- is $S(r)=\frac{-1}{ \ln r}$. Then, for $1<a<r<b$,
\begin{equation}
\label{eq:Rsortie}
 P_r[\tau(b)<\tau(a)] 
   =\frac{\ln(r/a)  \times \ln b}{\ln (b/a) \times\ln r}.
\end{equation}

We refer to section 2.1 in \cite{CP20} for more details on the many interesting properties of $\hW$ and $R$.

\subsection{Results for the future minimum}   \label{sec:res}

With $L(t)= \ln (t \vee 1)$ and $\ln (\cdot)$ the natural logarithm, define $ \ln_1 (t)=L(t), $ and for $k \geq 2, \ln_k(t) = L( \ln_{k-1}(t))$ so that $\ln_k(t) = (\ln \circ \ldots  \circ \ln) (t)$ for $t$ large.

\begin{theo}\label{th:critint} For $g: \R_+ \to \R_+$  non-increasing such that $(\ln t) g(\ln_2 t)$ is non-decreasing,
%%
%note: and converging to 0 at $\8$ est automatic
%%
we have:
\begin{equation}\label{eq:critint1}
\int^\8 g(u) du < \8 \implies {\ a.s.,\ } M(t) \geq e^{(\ln t) g(\ln_2t)}\quad  { eventually,}
\end{equation}
and
\begin{equation}\label{eq:critint2}
\int^\8 g(u) du = \8 \implies {\ a.s.,\ } M(t) \leq e^{(\ln t) g(\ln_2t)}\quad  {\it infinitely\ often.}
\end{equation}
\end{theo}
%%%%%%%%
(Note that the second assumption is quite natural in view of the monotonicity of $M(t)$.) 
%A direct consequence is:
%\begin{cor} \label{cor:Kov}
%For $g$ as in Theorem \ref{th:critint},
%%  For $g: \R_+ \to \R_+$  non-increasing such that $(\ln t) g(\ln_2 t)$ is non-decreasing,
%$\IP(  M(t) \leq e^{(\ln t) g(\ln_2t)}\; {i.o.})= 0 {\rm\ or\ } 1 $
%according to $\int^\8 g(u) du = \8$ or $< \8$.
%\end{cor}
Theorem \ref{cor:Kov} is a direct consequence of the above theorem. 
This result with an integral condition is reminiscent of Kolmogorov's test (see, e.g., sect. 4.12 in \cite{ImcK}), but the process $M$ here is not Markov.

These estimates are stronger than the corresponding ones in Th. 1.2 of \cite{PRU}. So are the following ones:

\begin{theo}\label{th:grandesvaleurs} There exist $0<K'<K<\8$ such that,  almost surely,
\begin{equation}\label{eq:grandesvaleurs1}
M(t) \leq K \sqrt{ t \ln_3t}\quad {\it eventually,}
\end{equation}
and
\begin{equation}\label{eq:grandesvaleurs2}
M(t) \geq K' \sqrt{ t \ln_3t}\quad  {\it infinitely\ often.}
\end{equation}
\end{theo}
%%%%%
%We  reformulate the result:
%%%%%%
%\begin{cor} \label{cor:OC}
%The limit 
%$$
%K^*=\limsup_{t \to \8} \frac{M(t)}{\sqrt{t \ln_3 t}} 
%$$
%is a.s. constant, and $0<K^*<\8$.
%\end{cor}
Theorem \ref{cor:OC} is essentially a reformulation of Theorem \ref{th:grandesvaleurs}, it will be proved below Remark \ref{rem:6.2}.

We recall the similar result \eqref{eq:limsupBES} for  transient Bessel processes:  a.s. for all
$a< \sqrt 2 < b$, the future minima process
$\min \{ \Bes^d(s); s \geq t\}$ is eventually smaller than $b \sqrt{ t \ln_2 t}$ and infinitely often larger than  $a \sqrt{ t \ln_2 t}$.
%Let us recall the similar result for    transient Bessel processes, Th. 4.1 in \cite{KLL94}: 
%$$
%  {\rm for} \; d >2, \quad \limsup_{t \to \8} \frac{ \min \{ \Bes^d(s); s \geq t\}}{\sqrt{ 2 t \ln_2 t}}=1.
%$$

Finally we mention that, for $d >2$,  $\min \{ \Bes^d(s); s \geq t\} \leq \eps \sqrt{ 2 t \ln_2 t}$ i.o., a.s.  for all $\eps>0$. (See \cite{KLL94}, P.349).

\subsection{Long time behavior of $R(t)$} \label{sec:LTB}

At large times the process $R$ behaves like $\Bes^2$. We emphasize that this is for the marginal law, but not for the future minimum. We formulate here precise statements of these facts.

It is well known that  the random variable $t^{-1/2}{\Bes}^2(t)$ converges to  the Rayleigh distribution %$\nu$ 
\begin{equation}\label{eq:hrayleigh}
d\nu(x) = x e^{-x^2/2} {\mathbf 1}_{(0,\8)}(x) dx 
\end{equation}
as $t \to \8$. Similarly for $R$, we have
%%%%%%%%%%%
\begin{theo}\label{thMarg1}
Let $Z \sim \nu$. As $t \to \8$,
\begin{equation}\nn
\frac{R(t)}{\sqrt{t}} \stackrel{\rm law}{\longrightarrow} Z \;.
\end{equation}
\end{theo}

\begin{theo}[Pointwise ergodic theorem]\label{thMarg2}
For all bounded continuous function $f$ on $(0,\8)$, as $t \to \8$,
\begin{equation}\label{eq:ergodic}
\frac 1t \int_0^{e^t-1} f\left( \frac{R(u)}{\sqrt{1+u}}\right) \frac{1}{1+u}  du \longrightarrow \int_\R f d\nu  \qquad a.s.
\end{equation}
\end{theo}

We will prove Theorems \ref{thMarg1}  and \ref{thMarg2} in section 4.

%%%%%%%%%%%%
\section{Regenerative structure} \label{sec:reg}

We fix a parameter $r>1$. We construct a regenerative structure associated with the process $R$ starting from $R(0)=1$.

\subsection{Renewal times}
We define a random sequence $(H_n, A_n, T_n)_{n\geq 0}$ by $ H_0, T_0=0, A_0=1, $ then
$$ 
\left\{ 
\begin{array}{ccl}
H_1&= & \inf\{ t > T_0: R(t)=r \}\\
A_1  &  = &    \inf \{ R(t); t \geq H_1\} \\
 T_1 & =  & \inf \{  t \geq H_1 : R(t)=A_1\}  \\
\end{array}
\right.
$$
and for $n \geq 1$, 
\begin{equation}\label{eq:renewal}
\left\{ 
\begin{array}{ccl}
H_{n+1}&= & \inf\{ t > T_n: R(t)=r A_n\}\\
A_{n+1}  &  = &    \inf \{ R(t); t \geq H_{n+1}\} \\
 T_{n+1} & =  & \inf \{  t \geq H_{n+1} : R(t)= A_{n+1}\}  
\end{array}
\right.
\end{equation}
Since $R$ is a continuous function with $\lim_{t \to \8} R(t)=\8$ a.s., we see by induction that $T_n < \8$ a.s. with 
$T_n < T_{n+1}$ and $\lim_{n \to \8} T_n=\8$ a.s. The  $T_n$ are not stopping times, but they are called {\it renewal 
times} for the following reasons.

\begin{prop}\label{prop:renewal} Let ${\cal G}_1= \sigma\big( T_1, (R(t) \1{t < T_1}; t \geq 0)\big)$. Then, 
$$ 
\left( \frac{R(T_1+A_1^2t)}{A_1}; t \geq 0\right) \quad {\it has\ same\  law \ as\ } R \; {\it and\ is\ independent\ of \ }  {\cal G}_1\;.
$$
\end{prop} 
%%%
This proposition is the building brick of the 
%%%%%
\begin{theo}\label{th:renewal} [Renewal structure]
The sequence 
$$
\left( \frac{R(T_n+A_n^2t)}{A_n}; t \in \left[ 0 , \frac{T_{n+1}-T_n}{A_n^2}\right]   \right)_{n \geq 0}
$$
is independent and identically distributed with the law of $(R(t); t \in [0,T_1])$.
\end{theo}
%%%%%%
 In particular, since $R(T_{n+1})=A_{n+1}$, the sequence 
 $$\left(  \frac{T_{n+1}-T_n}{A_n^2}, \frac{A_{n+1}}{A_n}\right)_{n\geq 0}$$
is i.i.d. and distributed as $(T_1, A_1)$. Therefore $(T_n,A_n)$ can be written using i.i.d.r.v.'s, which will be used repeatedly all through.
%%%%
\begin{proof} Proposition \ref{prop:renewal}. Recall that $P_r$ denotes the law of the process $R$  with $R(0)=r$. Observe that  $H_1$ is a stopping time, and denote by ${\cal F}_{H_1}$ the sigma-field of events that occur before time $H_1$.
By the strong Markov property,
$$
{\it under\ } P_1,\; (R(t+H_1))_{t \geq 0} \; {\it is\ independent\ of }{\cal F}_{H_1} \; {\it and\ has\ the\  law\ } P_r.
$$
Moreover, by Theorem 2.4 in \cite{Will74} (see also the proof of Lemma 3.9 in \cite{CP20}), conditionally on 
$T_1, (R(t); t \in [H_1,T_1])$ and $A_1=a$, 
$(R(T_1+t); t \geq 0)$ has the same law as $R$ starting from $a$ and conditioned to $R(t) \geq a, \forall t\geq 0$. 
By Brownian scaling,
the latter law is equal to that of $aR(\cdot/a^2)$ under $P_1$; see also Remark 2.5 in \cite{CP20}. Since ${\cal G}_1=\sigma( {\cal F}_{H_1};
(R(t);  t\in [H_1,T_1]))$ up to null events, we obtain the desired statement.
\end{proof}
\begin{proof} Theorem \ref{th:renewal}. 
By induction, Proposition \ref{prop:renewal} implies that for all $n$, the process
$\left( \frac{R(T_n+A_n^2t)}{A_n}; t \geq 0  \right)$ is independent of ${\cal G}_n=\sigma\big( T_n, (R(t) ; t < T_n)\big)
$ with the law of $R$. Then, the claim follows.
\end{proof}
As a direct consequence we have discovered a simple representation of crucial times and points of the process.
\begin{cor}\label{cor:renewal}
Define 
\begin{equation}\label{eq:def'} \nn
A_{n+1}'=\frac{A_{n+1}}{A_n} \;,\quad T_{n+1}'=\frac{T_{n+1}-T_n}{A_n^2} \;, \qquad n\geq 0.
\end{equation}
Then, $(A_n',T_n')_{n \geq 1}$ is an i.i.d. sequence with the same law as $(A_1,T_1)$, and we  have the representation
\begin{equation}\label{eq:rep}
\left\{
\begin{array}{ccl}
T_n&=&T_1'+A_1'^2\;T_2' + \ldots + (A_1'\ldots A_{n-1}')^2\;T_n'\\
A_n &=& A_1'\ldots A_n'
\end{array}
\right.\quad , \qquad n \geq 1.
\end{equation}
\end{cor}

\subsection{Description of a cycle}
Recall $r>1$ is fixed. We will shorten the notations: $(H,A,T)=(H_1,A_1,T_1)$.  Recall that $R$ starts from $R(0)=1$, hits $r$ at $H$ for the first time, and reaches its 
future minimum $A \in (1,r)$ at time $T$. We also introduce its maximum $B>r$ on the time interval $[H,T]$, as well as their logarithms
$U,V$:
\begin{equation}\nn
\left\{
\begin{array}{ccccc}
  A &  = & r^U &=& \min \{ R(t); t \geq H\} \\
 B & =  & r^V & = & \max \{ R(t); t \in [H, T]\}
\end{array}
\right.
\end{equation}
see figure \ref{fig:1}.
%%%%%
\begin{figure}
\begin{center}
\includegraphics[width=80mm, height=70mm]{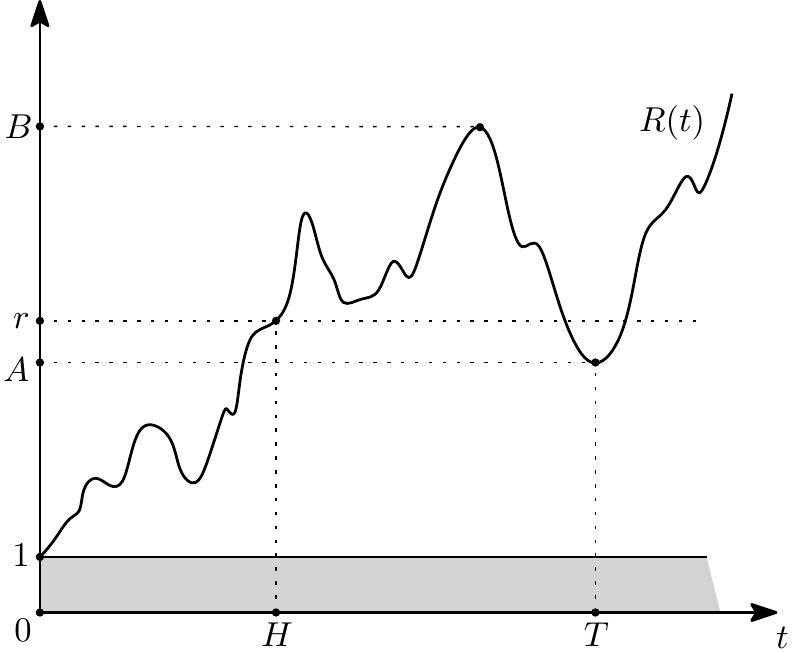}
\caption{First cycle: $A=r^U, B=r^V$}
\label{fig:1}
\end{center}
\end{figure}
It was shown in \cite{CP20} that $U$ is uniform on [0,1] (see \eqref{eq:Rsortie} with $b \to \8$), but we can even compute the joint law of $U$ and $V$.
For $1 < a-h< a< r<b$, we have by the strong Markov property
\begin{eqnarray}\nn
\IP\big(A\!\in\! [a\!-\!h,a),  B\!>\!b\big) 
&=&\IP\big(A\!\in\! [a\!-\!h,a),  B\!>\!b, \tau(b) \! < \! \tau(a)\big) + \IP\big(A\!\in\! [a\!-\!h,a),  B\!>\!b, \tau(b) \! > \! \tau(a)\big) \\ \nn
&=& P_r\big( \tau(b)<\tau(a)\big) \times P_{b}\big( \min\{R(t); t \geq 0\} \in [a-h,a)\big) + o(h)\\
&=& \frac{\ln (r/a)  \ln b}{\ln (b/a) \ln r} \times %%  \left(\frac{1}{a\ln b} h + o(h) \right) \nn
\frac{1}{a\ln b} h + o(h) \nn
\end{eqnarray}
using (2.16) in \cite{CP20}  and  that,
for $R$ started at $b$, $ \min\{R(t); t \geq 0\} $ has density $(a \ln b)^{-1}$ on $(1,b)$. 
Hence $(A,B)$ has a density
given by the negative of the $b$-derivative of the dominant term as $h \searrow 0$, i.e.,
$$
p_{A,B}(a,b) = \frac 1{ab \ln r}\; \frac{\ln (r/a)}{\ln^2(b/a)}\;,\quad 1<a<r<b.
$$
By changing variables, it follows that $(U,V)$ has density 
\begin{equation}\label{eq:denUV}
p_{U,V}(u,v) = \frac{1-u}{(v-u)^2} \1{0<u<1<v}
\end{equation}
We recover that $U$ is uniform on (0,1) and that $V$ has density
$$
p_V(v)=-\ln\big(1- 1/v\big)-1/v  \;, \quad v>1.
$$
It follows that for $v \geq 1$,
\begin{equation}
\label{eq:tailV}
\IP(V>v) = \sum_{n=1}^\8 \frac 1{n(n+1)v^n}  
\;,
\end{equation}
and then $\IP(V>v) \sim1/({2v})$ as $ v \to \8$.

%%%%%%%%%%%

\medskip

We also need information on the cycle length $T$.
For any $s\geq 1$ we consider the hitting time by $R$ starting at $s$ of its absolute minimum,
and denote by $\mu_s$ a r.v. with the same law:
\begin{equation} \nn %\label{def:mu_x}
\mu_s \sim P_s\big( \arg \min \{R(t); t \geq 0\} \in \cdot\big)
\end{equation}
Recall that, under $\IP$,  $R(0)=1$.\\
%%%
\begin{prop}\label{prop:mu_x} 
(i) We have $$ T= H + (T-H)\;,$$
where $ H $ and $ (T-H)$ are independent with $T-H \eqlaw \mu_r$.\\
(ii) For $u \in (0,1)$, the conditional law of $T$ given $U\geq u$ is equal to the %(unconditional) 
law of an independent sum 
$H + r^{2u} \mu_{(r^{1-u})}$.
\end{prop}
%%%%%%%
\begin{proof}
(i) directly follows from the strong Markov property for the Markov process $R$ and the stopping time $H$.

For (ii), we recall Remark 2.5 in \cite{CP20}: for $c>1$, denoting by $R^c$
the diffusion $R$ conditioned to stay outside $(1,c]$ and started at $c$, we have 
\begin{equation}\nn
R^c(\cdot) = c R(\cdot/c^2) \qquad {\rm in\ law.}
\end{equation}
(Alternatively, this follows from $R$ being the norm of conditioned Brownian motion \eqref{eq:MBconditionne} 
and from Brownian scaling.) Hence, for $s \in \R$, again from the strong Markov property,
\begin{eqnarray}\nn 
E_1[ e^{i s T } \big \vert U \geq u ] &=& E_1[ e^{i s (T-H + H) } \big \vert U \geq u ]
\\ &=&  E_1[ e^{i s  H} ] \times E_r[ e^{i s (T-H) } \big \vert U \geq u ] \nn \\ \nn
&=&  E_1[ e^{i s  H } ] \times E_r[ e^{i s \times \arg \min \{R(t); t \geq 0\} } \big \vert 
%\min_{[0,\8)} 
\min\{R(t); t \geq 0\} \geq r^u ] \\
\nn
&=& E_1[ e^{i s  H } ] \times E[ e^{i s r^{2u} \mu_{(r^{1-u})}}]
\end{eqnarray}
which proves the result.
\end{proof}
%%%%%%%%%%%%

\subsection{Tail estimates for $T$}
  
%  {\it c'est queue-de-T.tex et aussi sect.4.3 de "GVde$M_t$"}
  
  We need some estimates of the upper and lower tails of $T$, that we derive in this section. But first we state elementary comparisons of $R$ and Bessel processes, see \eqref{def:Bes}, that will be used all through the paper.
 \begin{prop} \label{prop:comp}  
(i) There exists a coupling of the processes $R$ and $\Bes^2$ starting at 1 such that 
\begin{equation}\nn   \forall t \geq 0, \quad
R(t) \geq \Bes^2(t) \;.
\end{equation}
(ii) For $\delta>0$ there exists a coupling of the processes $R$ and $\Bes^{2+\delta}$ starting at 1 such that for $\sigma= \sup \{ t \geq 0; R(t) \leq e^{2/\delta} \}$,
\begin{equation}\nn   \forall s \geq 0, \quad
R(\sigma+s) \leq \Bes^{2+\delta}(\sigma+s) - \Bes^{2+\delta}(\sigma) + e^{2/\delta}\;.
\end{equation}
\end{prop}
\begin{proof}
It is well known \cite{Ch00} that the stochastic differential equation \eqref{def:Bes} has a strong solution, so we can couple the processes $R$ and $\Bes^2, \Bes^{2+\delta}$ by driving equations \eqref{df_Rt} and \eqref{def:Bes} by the same Brownian motion $B$. Then, with $x^+=\max\{x,0\}$ for $x$ real, we have for all $t>0$ and all realization of $B$, 
\begin{eqnarray}\nn
d \big( \Bes^2(t) - R(t) \big)^+ &=& {\mathbf 1}_{\{   \Bes^2(t) \geq R(t)\}} \left( \frac 1{2\Bes^2(t)} - \frac 1{2R(t)} -\frac 1{R(t)\ln R(t)}\right) dt \\ \nn
&\leq& 0
\end{eqnarray}
which implies (i) by integration. Similarly for (ii) we write the differential
\begin{eqnarray}\nn
d \big( R(t) - \Bes^{2+ \delta}(t)  \big)^+ &=& {\mathbf 1}_{\{   \Bes^{2+\delta}(t) \leq R(t)\}} \left( \frac 1{2R(t)} +\frac 1{R(t)\ln R(t)} -  \frac {1+\delta}{2\Bes^{2+\delta}(t)} \right) dt \\ \nn
&\leq& 0 \qquad {\rm for\ } t \geq \sigma.
\end{eqnarray}
Integrating on $t\in [\sigma, \sigma+s]$ we obtain (ii). 
\end{proof}

We are now ready to start with the upper tail of $T$.

\begin{prop} \label{prop:tailT}  
As $t \to \8$,
\begin{equation}\label{eq:queuemine}
\IP( T \geq t) \sim \frac{\ln r}{\ln t}\;.
\end{equation}
More precisely, there exists constants $t_0$ and $C$ such that for all $t \geq t_0$, 
\begin{equation}\label{Queueminp}
    \left( 1- \frac{ \ln_3 t+C}{\ln t}\right) \frac{\ln r}{\ln t} \leq \mathbb{P}[T\geq t]\leq  \left( 1+ \frac{\ln_3 t+C}{\ln t}\right) \frac{\ln r}{\ln t}.
\end{equation}
\end{prop}
%%%%%%\subsection{Right tail  for $\mu_r$}
\begin{proof} We first obtain two preliminary estimates.\\

{\it Upper bound}: for $0<\eps<1$,
\begin{eqnarray}\nn
\IP(T\geq t)&=& \IP\left(T\geq t, V \geq \frac{\ln t}{2(1 \!+\! \eps)\ln r}\right) \!+\! \IP\left( T \geq t, V < \frac{\ln t}{2(1\!+\! \eps)\ln r}\right) \\ \nn
&\leq & \IP\left(V \geq \frac{\ln t}{2(1+\eps)\ln r}\right) + \IP\left( R(s) \leq t^{\frac{1}{2(1+\eps)}}, s \in [0,t]\right) \\ \label{eq:tailmu1}
&\leq & \frac {(1+\eps) \ln r}{\ln t} +\frac15 \left(\frac {2(1+\eps) \ln r}{\ln t}\right)^2 + 
C_0\exp\left(-C_1t^{\eps/(1+\eps)}\right) 
%\\ \nn &\leq& \frac {(1+\eps)\ln r}{\ln t} 
\end{eqnarray}
for $t \geq t_1$ with $t_1>0$ not depending on $\eps \in (0,1)$. Indeed, to obtain the first term we have used \eqref{eq:tailV}  in the form of $\IP(V\geq v) \leq (1/2v) + (1/5v^2)$ for large $v$. In order to obtain the second one, we first  bound $R(\cdot) \geq \Bes^{2}(\cdot)$, with $\Bes^2$ started at 0 using 
%that  the two diffusions have same constant diffusion coefficient, but  the last one has a smaller drift and has a strong solution (see \cite{Ch00})
Proposition \ref{prop:comp}, and finally that there exist positive
$C_0, C_1$ such that
\begin{equation}\label{eq:bmpo}
\forall t>0, \forall \rho > 0, \quad \IP \big( \Bes^{2}(s) \leq \rho, s \in [0,t]\big) \leq C_0 \exp\big(-C_1 \frac t{\rho^2} \big)\;,
\end{equation}
 see e.g. exercise 1 p.106 in \cite{Sz98}.

\medskip

{\it Lower bound}: for $0<\eps<1/2$,
\begin{eqnarray}\nn
\IP(T \geq t) &\geq& \IP\left(T-H \geq t, V \geq \frac{\ln t}{2(1-\eps)\ln r}  \right)\\ \nn
&=&  \IP\left( V > \frac{\ln t}{2(1-\eps)\ln r}  \right) - \IP\left(T-H \leq t, V \geq  \frac{\ln t}{2(1-\eps)\ln r}\right) \\ \nn
&\geq&  \IP \left( V > \frac{\ln t}{2(1-\eps)\ln r}  \right) - P_r\left( %\exists s \in [0,t]:  R(s) \geq  t^{\frac{1}{2(1-\eps)}}
\tau(R, t^{\frac{1}{2(1-\eps)}}) \leq t
\right)\\ \label{eq:tailmu2}
&\geq &  \frac {(1-\eps)\ln r}{\ln t}
- C_2\exp\left(-C_3t^{\eps/(1-\eps)}\right)
\end{eqnarray}
for $t \geq t_2$, with $t_2>0$ not depending on $\eps \in (0, \frac12)$. In \eqref{eq:tailmu2} we have used \eqref{eq:tailV} for the first term,  and we give details for the second one: 
for $|x|=r>1$ 
by \eqref{eq:MBconditionne}, 
% by Lemma  2.1 and formula (2.2)  in \cite{CP20}, denoting by $\tau^{\Bes^2}(\rho)$ the hitting time of $\rho$ by BES$^2$, 
we get for all $t>1$,
\begin{eqnarray}\nn
 P_r\big( \tau(R, t^{\frac{1}{2(1-\eps)}}) \leq t
\big) 
  &=&  
  \Psf_x \Big( \tau(|W|, t^{\frac{1}{2(1-\eps)}}) \leq t \;
   \big\vert\;  \tau (|W|,  t^{\frac 1{2\!(\!1\!-\!\eps)}}) \!<\!    \tau (|W|,1)  \Big)\\ \nn
  &\leq &
  \Psf_x\left(  \tau(|W|, t^{\frac{1}{2(1-\eps)}}) \leq t
   \right) \times \frac{ \ln t}{2(1-\eps) \ln r} 
  \\ \nn
  &\leq  & C_2\exp\left(-C_3 t^{\eps/(1-\eps)}\right)
\end{eqnarray}
 for some constant $C_2, C_3>0$ by  the moderate deviation principle for Brownian motion.
%%%%%

\medskip

For both the upper and lower bounds, we now choose $$\eps=\eps_t= \; \frac{\ln_3 t + C_4}{\ln t}$$ with a  constant $C_4$. Provided the constant $C_4$ is large enough, the terms $C_0\exp\left(-C_1t^{\eps_t/(1+\eps_t)}\right)$ and 
$C_2\exp(-C_3t^{\eps_t/(1-\eps_t)})$ are dominated by $(\ln t)^{-2}$.
We then get \eqref{Queueminp} from  \eqref{eq:tailmu1} and \eqref{eq:tailmu2}, taking any $C>C_4+\frac{4 \ln r}{5}$.

Finally, \eqref{eq:queuemine} is a direct consequence of \eqref{Queueminp}. The proof is complete.
\end{proof}
%%%%%
We also need to control the lower tail of $T$.
%%%
%
%
%
%%%%%%%%%%%%%%%%%%%%%
\begin{prop} \label{prop:LtailT}
(i) For all $\eps \in (0, r-1)$, there exists $t_0>0$ such that for $t \leq t_0$,
\begin{equation}\label{eq:LtailT0}
    \mathbb{P}[ T\leq t]
 %    =\mathbb{P}[ T\leq t_n r^{2i}]
    \leq  \exp\left(-\frac{(r-1-\eps)^2}{2t}\right)\;.
\end{equation}
(ii) 
For all $\eps>0$, there exists $t_1>0$ such that for $t \leq t_1$,
%There exist $t_0>0$ and $\rho>0$ such that for all $t \in (0, t_0]$ 
and all $u \in [0,1)$,
\begin{equation}\label{eq:LtailT}
\IP\big[ T \leq t | U \geq u\big] \geq  %\exp(-\frac \rho t)\;.
 \exp\left(-\frac{(r-1+\eps)^2}{2t}\right)\;.
\end{equation}
\end{prop}
%%%%%%%%%%%%%%%%%%%%%%%%%%%%%%%%%%%%%%%%%%%%%%%%%
\begin{proof} (i) Setting $a=1+\eps/2 \in (1,r)$ and using the strong Markov property for the hitting time of $a$ by $R$,
we obtain
\begin{eqnarray}\nn
\IP(T \leq t) &\leq & P_1( \tau(r)-\tau(a)  \leq t)\\ \nn
&=& P_a( \tau(r) \leq t)\\ \nn
& \stackrel{\eqref{eq:MBconditionne}}{=} & \Psf_{(a,0)}( \tau( |W|, r) \leq t \big\vert \tau(|W|;r)<\tau(|W|,1))  \\ \nn
&\leq&  \Psf_{(a,0)}( \tau( |W|, r) \leq t ) \times \frac{\ln r}{\ln a} \;.
%\\ \nn
%&\leq& 
%\IP(H \leq t) \\
%\nn &\leq& P_1( \Bes^{2} \; {\rm hits\ } r \; {\rm before\ } t  )\\ \nn 
%&\leq &  P_0( \Bes^{2} \; {\rm hits\ } r-1 \; {\rm before\ } t  )\\ \nn 
%&\leq & \exp\big(-\frac{(r-1)^2}{2t}(1+o(1))\big)\;.
\end{eqnarray}
Recalling large deviation results for Brownian motion in small time, e.g. section 6.8 of Ch. 5 in \cite{Az78}, 
\begin{equation} \label{eq:GDtempspetit}
\lim_{t \to 0} t \ln  \Psf_{(a,0)}( \tau( |W|, r) \leq t ) = - \frac{(r-a)^2}{2}\;,
\end{equation}
we see that the above upper bound implies (i).
%
%ICIII
%
%
%By comparison  $R \geq \Bes^{2}$,
%\begin{eqnarray}\nn
%\IP(T \leq t) &\leq& \IP(H \leq t) \\
%\nn &\leq& P_1( \Bes^{2} \; {\rm hits\ } r \; {\rm before\ } t  )\\ \nn 
%&\leq &  P_0( \Bes^{2} \; {\rm hits\ } r-1 \; {\rm before\ } t  )\\ \nn 
%&\leq & \exp\big(-\frac{(r-1)^2}{2t}(1+o(1))\big)\;.
%\end{eqnarray}

(ii) Let $t \leq 1$. 
By Proposition \ref{prop:mu_x}-(ii), and by comparing  $R $ and $\Bes^{2}$ from Proposition \ref{prop:comp} (i), we obtain
\begin{eqnarray}\nn
\IP( T \leq t | U \geq u) & \geq & \IP( H \leq t -t^2) \times \IP( r^{2u} \mu_{(r^{1-u})}\leq t^2 ) \\
&\geq& \IP( \Bes^{2} (t-t^2) \geq r) \times  \IP\left(\mu_{(r^{1-u})}\leq \frac{t^2}{ r^{2u} } \right) 
\nn \\
&=& %e^{-C\frac{r^2}{t}} \times   %
\Psf_{(1,0)} \left(  |W(t-t^2)|\geq r \right) \times \nn \\ && \qquad \qquad 
P_{r^{1-u}}\Big( \arg \min \{R(s); s \geq 0\} \leq \theta \Big)      \;,
\label{eq:20j1}
\end{eqnarray}
%by comparison with BES$^{2}$.  
with $\theta=\frac{t^2}{ r^{2u} }$. We estimate  the first term using again 
 large deviation for Brownian motion in small time \cite{Az78}: for $|x| < r$,
\begin{equation} \label{eq:GDtempspetit0}
\lim_{t \to 0} t \ln  \Psf_x( |W|(t) \geq r ) = - \frac{(r-|x|)^2}{2}\;.
\end{equation}
To estimate the second term in \eqref{eq:20j1}, note that $R( \theta) \geq r^{1-u}+ \sqrt \theta$ and $R(s)\geq r^{1-u}$ for all $s \geq \theta$  implies that,  $P_{r^{1-u}}$-a.s., 
$R$ achieves its minimum before time $\theta$. Hence, by Markov property and \eqref{eq:Rsortie},
\begin{eqnarray}
P_{r^{1\!-\!u}}\Big( \arg \min \{R(s); s \geq 0\} \leq \theta \Big)  \!   & \! \geq \! & \!
P_{r^{1\!-\!u}}\Big( R(  \theta) \geq r^{1\!-\!u}+ \sqrt \theta \Big) \! \times  \!
\left(1- \frac{\ln r^{1\!-\!u}}{\ln ( r^{1\!-\!u}+ \sqrt \theta)} \right)  \nn \\
 & \! \geq \! & \!
\IP\left( B(\theta) \geq \sqrt{\theta} \right) \! \times  \!
\left(1- \frac{\ln r^{1\!-\!u}}{\ln ( r^{1\!-\!u}+ \sqrt \theta)} \right)  \nn
\\ \nn &\geq& 
\IP\left( B(1) \geq 1 \right) \! \times  \!  \frac{t}{2r \ln r}  \qquad {\rm for \ small \ } t,
\end{eqnarray}
arguing on the second line that $R$ dominates Brownian motion by comparing the drift.
%
%arguing that $R$ dominates Brownian motion by comparing the drift.
%Now, for $t$ small, 
%we have $\left(1- \frac{\ln r^{1\!-\!u}}{\ln ( r^{1\!-\!u}+ \sqrt \theta)} \right)  \geq \frac{t}{2r}$.
%In view of \eqref{eq:20j1}, the proof is complete.
Combined with \eqref{eq:20j1} and \eqref{eq:GDtempspetit0}, this completes the proof of (ii).
\end{proof}
%%%%%%%%%%
\subsection{Tail estimate for $U$}
Recall Hoeffding's inequality \cite{H63}, or Th. 2.8 in \cite{BLM13}:
for  $b<1$, $c>1$ and $i \geq 1$,
\begin{equation}\label{eq:hoeffdingg}
\mathbb{P}\left[2(U_1\!+\!\dots\!+\!U_i)\geq c.i\right] \leq \exp\big(-\frac  i2 (c-1)^2\big),
\end{equation}
and
\begin{equation}\label{eq:hoeffdingl}
\mathbb{P}\left[2(U_1\!+\!\dots\!+\!U_i)\leq b.i\right] \leq \exp\big(- \frac i2 (1-b)^2\big).
\end{equation}
%%%
\begin{rem}[The random difference equation \eqref{eq:ar}]\label{rem:RDE}
Introduce the sequence 
$$S_n=\frac{T_n}{A_n^2}$$
which is key in Section \ref{sec:grandesvaleurs}. 
In view of \eqref{eq:rep}, we see that it solves the recursion 
\begin{equation}\nn
S_{n+1}= \alpha_{n+1} S_n + \beta_{n+1}
\end{equation}
(i.e.,  \eqref{eq:ar} above), with
$$
\alpha_n=(A_n')^{-2}\;,\quad \beta_n= \frac {T_n'}{(A_n')^{2}}\;.
$$
The bi-dimensional sequence $(\alpha_n,\beta_n), n \geq 1$, is i.i.d., and falls into the usual setup of random difference equation. In our case, the following quantities exist 
$$
a:= \E[ \ln \alpha_1]  \;, \quad b:= \lim_{t \to \8}  \IP[ \beta_1 > t]\times {\ln t} \;,
$$ 
and satisfy $a<0$ (contractive case), $0<b<\8$ (very heavy tail). Following \cite{Alsmeyer}  and \cite{Zerner18}, this prevents the Markov chain $S_n$ to be positive recurrent: though the contraction brings stability to the process, 
yet occasional large values of $\beta_n$ overcompensate this behavior so that positive recurrence fails to hold. 
In our case, we easily check from \eqref{eq:queuemine} that
$$
b = -a \quad (= \ln r)
$$
in which case the Markov chain $S_n$ is null recurrent, but in a critical manner: the chain is transient if $b>-a$ and null recurrent if $b \leq -a$.
\end{rem}

\section{Proofs for section \ref{sec:LTB}} \label{sec:proofsLTB}

We consider the process $R$ from \eqref{df_Rt} on a geometric scale,
\begin{equation} \label{eq:h1}
X(t)= e^{-t/2} R(e^t\!-\!1)
\end{equation}
and we observe that
$$ 
\beta(t)= \int_0^{e^t\!-\!1} \frac 1{\sqrt{1+s}} \;dB(s)
$$
is a standard Brownian motion by Paul L\'evy's characterization. We claim that  $X$ solves the stochastic differential equation 
\begin{equation}\label{eq:hnh}
\left\{
\begin{array}{ccl}
dX(t)&=& \left(   {\displaystyle  \frac{\dis 1}{\dis 2X(t)}  - \frac{\dis X(t)}{\dis 2} + \frac{\dis 1}{\dis X(t) \ln [ e^{t/2} X(t)] } } \right) dt + d\beta(t)\\  X(0)&=&R(0)\;.
\end{array}
\right.
\end{equation}
Indeed,
\begin{eqnarray*}
X(t)&=&  e^{-t/2} X(0) + e^{-t/2}  \int_0^{e^t\!-\!1} \big(\frac{1}{2R}+ \frac{1}{R \ln R}\big)\!(s)\; ds + e^{-t/2}  B({e^t\!-\!1}) \\&=& J(t) +  K(t) + L(t)\;,
\end{eqnarray*}
with $dJ(t)=-\frac 12 J(t) dt$, and 
\begin{eqnarray*}
\frac{dK(t)}{dt} &=& -\frac 12 K(t) + \frac 1{2 X(t)} + \frac 1{X(t) \ln [e^{t/2} X(t)]} \;,  \\
dL(t) &=& -\frac 12 L(t) dt  + e^{-t/2} dB(e^t\!-\!1)\;.
\end{eqnarray*}
Moreover, we easily check the equality 
$$
\int_0^t  e^{-s/2} dB(e^s\!-\!1) = \int_0^{e^t\!-\!1}  \frac 1{\sqrt{1+u}} \;dB(u)
$$
in the Gaussian space generated by $B$. Adding up terms, we see that $X$ solves the stochastic differential equation  \eqref{eq:hnh}.
%%%%%%%%
Denote by $b_t$, resp. $b_\8$ the drift coefficient and its limit, given for $ x \in (0,\8)$ by
\begin{equation}\nn
b_t(x)= \frac{1}{2x} - \frac x2 + \frac{1}{x (\ln x+ t/2)}\,, \qquad b_\8(x)= \frac{1}{2x} - \frac x2 \, ,
\end{equation}
and by $X^{(\8)}$ the homogeneous diffusion
\begin{equation}\label{eq:hh}
dX^{(\8) }(t) = \left( \frac{1}{2X^{(\8)}(t)} - \frac{X^{(\8)}(t)}{2} \right) dt + d\beta(t)\;.
\end{equation}
Following the approach of Takeyama \cite{T84}, we state the following

\begin{lem}\label{lem:ashom}
The diffusion $X(t)= e^{-t/2} R(e^t\!-\!1)$ is asymptotically homogeneous with homogeneous limit $X^{(\8)}$, i.e, for all continuous $f$ with compact support in $(0,\8)$ and all $t>0$, 
$$ E\big[ f(X(t+s)) | X(s)=x \big]  \longrightarrow E_x\big[ f(X^{(\8)}(t))\big]\quad {\rm as\ } s \to \8$$
uniformly on compact subsets of $(0,\8)$.
\end{lem}
\begin{proof} It is easier to consider $ \widehat X(t) = X(t)-e^{-t/2}$ which takes values in the fixed interval $(0,\8)$, and $\widehat{ X^{(s)}}(t)=\widehat X (s+t)$. 
Then,  the coefficients of the diffusion $\widehat{ X^{(s)}} $ converge to those of $X^{(\infty)}$, uniformly  on compact subsets of $(0,\8)$, and the corresponding martingale problems have a unique solution. Thus, 
Theorem 11.1.4 in  \cite{SV06}  yields  the desired result.
% SV conclut bien a la cv localement uniforme!   %%Need to argue $E X(t)$ is upper bounded and that $X(t)$ is bounded away from 0 in probability.
\end{proof}

The process $X^{(\8)}$ is the transform $X^{(\8)}(t)=X^{(\8,2)}(t)= e^{-t/2} \Bes^2(e^t-1)$
of $\Bes^{2}$ by the rescaling and deterministic time-change \eqref{eq:h1}. It is recurrent and ergodic on $(0,\8)$ with the Rayleigh law as invariant probability measure,
\begin{equation}\nn
d\nu(x) = x e^{-x^2/2} {\mathbf 1}_{(0,\8)}(x) dx 
\end{equation}

A first consequence is that $R$ marginally behaves like $\Bes^2$.

\begin{cor}[Convergence in law] \label{prop:cvlaw}
Let $Z \sim \nu$. As $t \to \8$,
\begin{equation}\nn
\frac{R(t)}{\sqrt{t}} \stackrel{\rm law}{\longrightarrow} Z \;.
\end{equation}
\end{cor}

%\begin{proof}
%We combine convergence to equilibrium $X^{(\8)}(t) \to Z$ in law as $t \to \8$ and Lemma \ref{lem:ashom}. [A COMPLETER Reference ??] %%Need to argue $E X(t)$ is upper bounded and that $X(t)$ is bounded away from 0 in probability.
%\end{proof}
%%%%%%%%
%{\color {blue}
\begin{proof}
Denote by $P_{s,t}, P_{s,t}^{(\8)} (0 \leq s \leq t)$ the Markov  semi-groups associated to $X$ and $X^{(\8)}$,
$$
(P_{s,t}f) (x) = E\big[ f(X(t)) | X(s)=x \big]\;,\quad (P_{s,t}^{(\8)} f)(x)   = E\big[ f(X^{(\8)}(t)) | X^{(\8)}(s)=x \big]\;,
$$
so that $P_{s,t}^{(\8)}=P_{0,t-s}^{(\8)}$. For a bounded continuous $f: (0,\8) \to \R$ we write
\[ \begin{split}
P_{0,t+s}f (x) - &\int f d\nu = 
P_{0,s} (P_{s,s+t}f)(x)- \int f d\nu \\
=& P_{0,s} \left(P_{s,s+t} f  - P_{s,s+t}^{(\8)}f\right) (x)+
P_{0,s} \left(P_{s,s+t}^{(\8)}f- \int f d\nu \right)(x) \;,
\end{split}
\]
%\begin{eqnarray}\nn
%P_{0,t+s}f (x) - \int f d\nu &=& 
%P_{0,s} (P_{s,s+t}f)(x)- \int f d\nu \\ \nn
%&=&
%P_{0,s} \left(P_{s,s+t} f  - P_{s,s+t}^{(\8)}f\right) (x)+
%P_{0,s} \left(P_{s,s+t}^{(\8)}f- \int f d\nu \right)(x) \;,
%\end{eqnarray}
where both terms vanish as $s, t \to \8$, which is our claim. Indeed, by convergence of $X^{(\8)}$ to equilibrium, 
$P_{s,s+t}^{(\8)}f - \int f d\nu= P_{0,t}^{(\8)}f - \int f d\nu \to 0$ uniformly on compact subsets of $(0,\8)$  as $t \to \8$ and 
Lemma \ref{lem:ashom} implies that $P_{s,s+t} f  - P_{s,s+t}^{(\8)}f \to 0$ uniformly on compact  as $s \to \8$: thus, we only 
need to prove tightness, i.e. that for all positive $x$, 
\begin{equation}
\inf \big\{ P_{0,s} ( \mathbf 1_{[\eps, 1/\eps]}) (x) ; s \geq 1\big\}   \to 1 \quad {\rm as} \; \eps\to 0\;.
\end{equation}
But this follows from the next two bounds 
\begin{itemize}
\item $R \geq \Bes^2$ (see Proposition \ref{prop:comp} (i)) which implies  that $X \geq X^{(\8)}$ ,
\item  $\sup_{s\geq 1} \E X(s)^2 \leq 
\sup_{s \geq 1}  s^{-1}\E R(s)^2 < \8$ that we explain now. 
\end{itemize}
First recall from \cite{CP20} that $\frac 1{\ln R}$ is a martingale, and so, for all $r>1$, 
\begin{equation} \label{eq:mart9}
E_r \left[ \frac 1{\ln R(t)} \right]= \frac 1{\ln r}\;.
\end{equation}
By It\^o's formula, 
\begin{equation}\label{eq:ito9}
d(R^2)= 2\Big( 1 + \frac {1}{\ln R(t)} \Big) dt + 2 R(t) dB(t)\;.
\end{equation}
Thus, for all $r>1$,
$$E_r [ R(t)^2] = r^2 + 2t \Big(1 + \frac 1{\ln r}\Big)\;.$$
We now consider the process starting from $R(0)=1$. Integrating \eqref{eq:ito9}, we get
\begin{eqnarray}\nn
E_1\big[  R(t)^2 {\mathbf 1}_{\tau(r)<t} \big] 
&=& 2 E_1\left[  \int_0^t  {\mathbf 1}_{\tau(r)<s}  \big( 1 + \frac {1}{\ln R(s)} \big) ds + \int_0^t  {\mathbf 1}_{\tau(r)<s}  R(s) dB(s) \right] \\ \nn
&\stackrel{\rm Markov}{=}&  2  \int_0^t  E_1\left[  {\mathbf 1}_{\tau(r)<s}  E_r \left( 1 + \frac {1}{\ln R(\cdot)} \right)_{\cdot=s-\tau(r)}\right] ds  +0 \\ \nn
&=&  2   \Big(1 + \frac 1{\ln r}\Big)   E_1\big[\big(t-\tau(r)\big)^+\big]   \nn
\end{eqnarray}
by \eqref{eq:mart9}. Finally we obtain that 
$$\E R(t)^2 
= \E\big[  R(t)^2 {\mathbf 1}_{\tau(r)<t} \big]  + \E\big[  R(t)^2 {\mathbf 1}_{\tau(r)<t} \big]  
\leq r^2 + 2t \left(1 + \frac 1{\ln r}\right)$$ for all $r>1$. The corollary is proved.
\end{proof}

\begin{cor}[Pointwise ergodic theorem]
For all bounded continuous $f$ on $(0,\8)$, as $t \to \8$,
\begin{equation}\nn
\frac 1t \int_0^t f(X(s)) ds \longrightarrow \int_\R f d\nu  \qquad a.s.,
\end{equation}
or, equivalently,
\begin{equation}\nn
\frac 1t \int_0^{e^t-1} f\left( \frac{R(u)}{\sqrt{1+u}}\right) \frac{1}{1+u}  du \longrightarrow \int_\R f d\nu  \qquad a.s.
\end{equation}
\end{cor}

\begin{proof}
It is easy to check that,  w.l.o.g.,  we can assume that $f: (0,\8) \to \R $ is non-decreasing.
By the comparison principles of  Proposition \ref{prop:comp}, we can couple the processes $R, \Bes^2, \Bes^{2+\delta}$ ($\delta>0$)
starting at 1 such, a.s., for all $t \geq \ln (1+\sigma)$ with 
$$\sigma= \sup\{ s>0: R(s) \leq e^{2/\delta}\} < \8, $$ we have
$$X^{(\8,2)}(t) \leq X(t)  \leq X^{(\8,2+\delta)}(t) -  e^{-t/2}( \Bes^{2+\delta}(\sigma) - e^{2/\delta})
\;.$$
By the pointwise ergodic theorem for $X^{(\8,2)}$ and $ X^{(\8,2+\delta)}$ and monotonicity of $f$, we derive
$$
\int f d\nu \leq \liminf_{t \to \8} \frac 1t \int_0^t f(X(s)) ds \leq \limsup_{t \to \8} \frac 1t \int_0^t f(X(s)) ds \leq 
\int f d\nu_\delta\;,
$$
where $d\nu_\delta(x)= c_{\delta} x^{1+\delta/2} e^{-x^2/2} {\mathbf 1}_{(0,\8)}(x) dx $ is the invariant law of $ X^{(\8,2+\delta)}$.
As $\delta$ vanishes, the two extreme members coincide, ending the proof of the first statement. The second one follows by changing variables.
\end{proof}
%%%%%%%%%%%%
\section{Proof of Theorem \ref{th:critint}} \label{sec:prth:crintint}
%{\it Introduire $r^\pm$ et les GD} [??]

Recall the representation \eqref{eq:rep} from Corollary \ref{cor:renewal},
\begin{equation}\nn
T_k=
T_1'+A_1'^2T_2' + \ldots + (A_1'\ldots A_{k-1}')^2T_k'\;,\qquad 
A_k=A_1'\ldots A_k'
\end{equation}
with  $(T_k', A_k')_{k \geq 1}$  an i.i.d. sequence with the same law as $(T_1,A_1)$. 
\medskip

Fix $r_\pm$ with $1<r_-<r<r^+<\8$. By \eqref{eq:hoeffdingg} and \eqref{eq:hoeffdingl}, 
%Cram\'er's theorem \cite{DZ},  
with probability one there exists some finite random $k_0$ such that for all $k\geq k_0$
$$
 r_-^{k/2} \leq  A_1'\ldots A_k' = r^{U_1+\ldots + U_k} \leq   r_+^{k/2} \;.
$$
In what follows we will use the rough bounds 
 \begin{equation}\label{eq:bds}
 \max_{i=1, \dots, k} %r_-^{i-1} 
 T_i'\leq T_k \leq T_{k_0}+ (k-k_0) \max_{i=1, \dots, k} r_+^{i-1} T_i'\;.
\end{equation}
%%%%%%%%%
\begin{lem} \label{lem:1...}There exists a constant $c$ such that for all sequence $(\delta(k))_k$ tending to 0, we have
$$
\IP \Big[k \max_{i=1, \dots, k} r_+^{i-1} T_i' \geq  e^{k/\delta(k)}\Big] \leq c \delta(k)
$$
eventually. %for all  large enough $k$.
\end{lem}
%%%%%%%%%%
\begin{proof}
Fix $a $ with $1<a<e$.      Letting $v_k=a^{\frac{k}{\delta(k)}}$ and $t_k=k r_+^k v_k$, we note that $e^{\frac{k}{\delta(k)}} \geq t_k$ eventually since  $\delta$ vanishes, and we have by independence
    \begin{align}\nn
 \IP[k \max_{i=1, \dots, k} r_+^{i-1} T_i' < t_k]  
%        & = \IP[\max_{i\leq k} r_+^{i-1} T_i' < r_+^k v_k] \\
        & = \Pi_{i=1, \dots, k} \IP[T_i'< r_+^{k-i+1} v_k] \nn
    \end{align}
%%%%%
From Proposition \ref{prop:tailT}  there exists $c_1>0$ such that for all $t>1$
$$
\IP(T_1 \geq t) \leq \frac{c_1}{\ln t}
$$
and since $v_k \to \8$ as $k\to \8$, we have for all  large enough $k$,
    \begin{align} \nn
     \IP[k \max_{i=1, \dots, k} r_+^{i-1} T_i' < t_k]  
        & \geq \Pi_{i=1}^k \left(1-\frac{c_1}{\ln (r_+^{k-i+1} v_k)}\right)\\ \nn
        & = \Pi_{i=1}^k \left(1-\frac{c_1}{\ln (r_+^i v_k)}\right)\\ \nn
        & \geq \exp \left(-2c_1 \sum_{i=1}^k \frac{1}{i \ln r_+ + \ln v_k}\right)\\ \nn
        & \geq \exp \left(-\frac{2c_1}{\ln r_+} \ln\left(\frac{k \ln r_+ +\ln v_k}{\ln v_k}\right) \right)\\ \nn
        & = \exp \left( - \frac{2c_1}{\ln r_+} \ln(1+  \frac{\ln r_+}{\ln a} \delta(k)) \right)\\ \nn
        & \geq 1-c \delta(k)
    \end{align}
 with $c = 2 c_1/\ln a$ for all large $k$, since $\delta$ vanishes at $\8$. This ends the proof.
\end{proof}
%%%%%%%%%
\begin{proof}{\it Theorem \ref{th:critint},  claim \eqref{eq:critint1}.}
 Let 
 $$\delta(t)=g(\ln t), \qquad \ka(i)=2^i , i\geq1, \qquad 
  \Ka=\{\ka(i) : i\geq 1\}.$$ 
  Define, for $x \geq 2$, $\left\lfloor x \right\rfloor_\Ka=\max \{k\in \Ka : k\leq x\}= 2^{\lfloor (\ln x)/(\ln 2) \rfloor}$. Note that
  \begin{equation} \label{eq:1...}
x \geq \left\lfloor x \right\rfloor_\Ka \geq  x/2 \;. %\frac{x}{2}\;.
\end{equation}
First, since $g$ is non-increasing,
    \begin{equation} \label{eq:serie-integrale} 
        \begin{split}  
            \sum_{k\in \Ka} \delta(k) & = \sum_{i\geq 1} \delta(k(i)) \\
            & = \sum_{i\geq 1} g(\ln k(i)) \\
            & = \sum_{i\geq 1} g(i  \ln 2) \\
            & \leq \frac 1 {\ln 2} \sum_{i\geq 1} \int_{(i-1)\ln 2}^{i \ln 2} g(t) {d}t\\
            & =\frac 1 {\ln 2} \int_0^{\infty} g(t) {d}t < \infty
        \end{split}
    \end{equation}
Fix a constant $c_2 >0$ to be chosen later and $c_3=c_2^{-1}$. Combining Borel-Cantelli's lemma and Lemma \ref{lem:1...},  we have a.s.
$$
k \max_{i=1, \dots, k} r_+^{i-1} T_i' <  e^{c_2 k/\delta(k)} \quad {\it for\ all\ } k \in \Ka  {\it \ large \ enough},
$$
and, in addition to \eqref{eq:bds}, we have for large $k \in \Ka$,
\begin{equation} \label{eq:vazy}
T_k \leq T_{k_0} + \frac{k-k_0}{ k} e^{c_2 k/\delta(k)} \leq e^{c_2 k/\delta(k)} \, % {\it ev.\ on\ } \Ka 
\end{equation}
since $g$ is non-increasing. By integrability, $g$ is vanishing at infinity, so the function 
$$f(t)= c_3 (\ln t )\; g(\ln_2 t)$$
is such that $f(t) \leq \ln t$ eventually, and also $g(\ln_2 t)\leq g(\ln f(t))$ by monotonicity. Thus, for large
$k$ and $t$'s,  
\begin{equation}\label{eq:implies}
 k \leq c_3 (\ln t) \delta(\ln t)=f(t) \quad
 {\rm implies\; that} \qquad \qquad \qquad \qquad \qquad \qquad \phantom{} 
\end{equation}
$$
\phantom{} \qquad \qquad \qquad \qquad \qquad 
    \frac{k}{\delta(k)} =\frac{k}{g(\ln k)}\leq \frac{f(t)}{g(\ln f(t))} = \frac{c_3(\ln t) g(\ln_2 t)}{g(\ln f(t))}\leq c_3 \ln t.
$$
%%%%%%%%%
Now, define random integers  $k(t)= \max \{k\in \Ka ; T_k \leq t\}$, and note from \eqref{eq:vazy} that a.s., for large $t$ 
we have $k(t)\geq \max\{k \in \Ka ; e^{c_2\frac{k}{\delta(k)}}\leq t\}$. Then,  a.s., for all large enough $t$,
\begin{align} \nn
    M_t \geq M_{T_{k(t)}} = A_{k(t)} \geq r_-^{\frac{k(t)}{2}} & \geq r_-^{\frac{1}{2} \max\{k \in \Ka: e^{c_2\frac{k}{\delta(k)}}\leq t\}}\\ \nn
    & = r_-^{\frac{1}{2} \max\{k \in \Ka : \frac{k}{\delta(k)}\leq c_3 \ln t\}}\qquad ({\rm using\ } c_3=c_2^{-1})\\ \nn
    & \geq r_-^{\frac{1}{2} \max\{k \in \Ka : k\leq f(t)\}} \qquad \qquad {\rm (by\ \eqref{eq:implies})}\\ \nn 
    &= r_-^{\frac{1}{2} \left\lfloor c_3(\ln t) \delta(\ln t)\right\rfloor_\Ka}\\ \nn
    &\geq r_-^{\frac{c_3}{4} (\ln t) \delta(\ln t)} \qquad \qquad    {\rm (by\ \eqref {eq:1...})} 
\end{align}
Taking $c_3=c_2^{-1}> 4/ \ln r_-$, we conclude that
a.s., $M(t) \geq e^{(\ln t) g(\ln_2t)}$  { eventually,}
ending the proof of \eqref{eq:critint1}.
\end{proof}

We now turn to the proof of claim
\eqref{eq:critint2} of Theorem \ref{th:critint}.
We start with a lemma:

\begin{lem}\label{lem:petitptag}
Let  $(n_k)_{k\geq0}$ be a non-decreasing  sequence of integers and $(t_k)_{k\geq0}$ be a sequence with $t_k > 1$. Then, 
$$\sum_{k\geq0} \frac{n_{k+1}-n_k}{\ln t_{k+1}}=\infty \quad  \Longrightarrow  \quad  a.s., \; T_{n_k}\geq t_k \; {\it infinitely\ often}.$$
\end{lem}

\begin{proof} 
The events $E_k=\{\max_{i=n_k+1,\ldots,n_{k+1}} T_i' \geq t_{k+1}\} , k \geq 0$ are independent with $E_k \subset \{ T_{n_{k+1}}\geq t_{k+1}\}$. Hence the conclusion holds as soon as these events occurs  infinitely often a.s. By the second Borel-Cantelli lemma, it suffices to show that the assumption implies $\sum_{k\geq0} {\IP}(E_k)=\infty.$
 We use Proposition \ref{prop:tailT} and  independence. The case when $t_k$ does not tend to infinity  is easily considered, so we assume from now on that $k$ is large enough so that 
 $\IP(T\geq t_{k+1}) \geq c/ \ln t_{k+1}$ for some fixed constant $c \in (0,\ln r)$. Then, we can bound
 \begin{align} \nn
    {\IP}(E_k) & = 1-\IP(T\leq t_{k+1})^{n_{k+1}-n_k} \\ \nn
        & \geq 1 - \left(1 -\frac{c}{\ln t_{k+1}}\right)^{n_{k+1}-n_k} \\  \nn
        & \geq 1- \exp\left(- \frac{c(n_{k+1}-n_k)}{ \ln t_{k+1}} \right) 
%        \\   & = \frac{(n_{k+1}-n_k)\ln r}{\ln t_{k+1}} +o\left(\frac{n_{k+1}-n_k}{\ln t_{k+1}}\right)
    \end{align}
 which is the general term of a divergent series.
\end{proof}
%%%%%
\begin{proof} {\it Theorem \ref{th:critint}, claim
\eqref{eq:critint2}.}
%We finish the proof of \eqref{eq:critint2}. 
Let  us consider  
$$t_k =e^{e^k}\;,\quad n_k=\lfloor f(t_k)\rfloor\;,\quad f(t)= c_3 (\ln t) g(\ln_2 t)$$ with $c_3>0$ to be fixed later. 
Note that $f$ is non-decreasing by assumption. We have 
   \begin{equation}\nn
   \begin{split}\nn
              \sum_{k\geq0} \frac{n_{k+1}-n_k}{\ln t_{k+1}} 
         = \sum_{k\geq0} \frac{\lfloor f(t_{k+1})\rfloor -\lfloor f(t_k)\rfloor }{\ln(t_{k+1})} 
         &= \sum_{k\geq0} \frac{f(t_{k+1})-f(t_k)}{\ln(t_{k+1})} + c_4\\
           & = c_3 \sum_{k\geq0} g(k+1)-\frac{1}{e} g(k) + c_4 \nn
            \end{split}
    \end{equation}
with a constant  $c_4$ which is finite since $t_k$ is increasing fast and the truncation error is bounded.
%, and where the inequality comes from monotonicity of the sequence $t_k$.
As in \eqref{eq:serie-integrale},
$ \sum_{k\geq 0} g(k)\geq  \int_0^{\infty} g(t) {d}t = \infty,$ 
and
    $$\sum_{k=0}^n g(k+1)-\frac{1}{e} g(k) = g(n+1)-\frac{1}{e}g(0) +\left(1-\frac{1}{e}\right)\sum_{k=1}^n g(k).$$
Therefore $\sum_{k\geq0} \frac{n_{k+1}-n_k}{\ln t_{k+1}}=\infty$. From Lemma \ref{lem:petitptag} we obtain
that a.s.,  $T_{n_k}\geq t_k$ i.o., which shows that
$$
M_{t_k}\leq M_{T_{n_k}} = A_{n_k}\leq r_+^{n_k}\leq r_+^{f(t_k)}\;.
$$
Taking  $c_3< 1/\ln r_+$, we obtain the desired claim.
\end{proof}

%We end the section with a remark. 
%
%\begin{rem}
%The non-increasing property of $g$ is not essential, it is only used to derive the convergence of $\sum g(i)$ from that of $\int^\8 g(x) dx$.
%Many natural examples of test functions $g$ have this mononicity property. 
%%On the other hand, the natural monotonicity assumption is to require $(\ln t) g(\ln_2t)$ to be non-decreasing.
%\end{rem}

%%\break
%%%%%%%%%%%%
\section{Proof of Theorem \ref{th:grandesvaleurs}} \label{sec:grandesvaleurs}
\addtocounter{equation}{1}   
%%%
%
%  Why do I need that ?
%
%%%
We study the sequence 
\begin{equation*} \label{eq:test}
S_n=\frac{T_n}{A_n^2}=\sum_{i=1}^n  \frac{T_i' A_{i-1}^2}{A_n^2}=
\sum_{i=1}^n  \frac{T_i'}{r^{2(U_i+\dots+U_n)}}\;,
\end{equation*}
which can be written in the form
\begin{equation} \label{eq:Sn=Snm}
S_m=\frac{S_n}{r^{2(U_{n+1}+\dots + U_m)}}+S_{n+1}^m,
\end{equation}
where, for $1\leq n < m$,  
$$S_{n+1}^m = \sum_{i=n+1}^m  \frac{T_i'}{r^{2(U_i+\dots+U_m)}}.$$
The point is that, in \eqref{eq:Sn=Snm}, $S_n$ and $S_{n+1}^m$ are independent, with $S_{n+1}^m$ equal to
 $S_{m-n}$ in law.
 \medskip
 
 We study the convergence/divergence of the series $\sum_{n\geq 1} \mathbb{P}[S_n\leq t_n]$, 
 with $t_n$ of the form 
 \begin{equation}\label{eq:tn}
t_n= \frac{\beta}{\ln_2 n} \wedge 1
\end{equation}
 for some $\beta >0$.
 %for some non-increasing, positive sequence $(t_n)_{n\geq 1}$. 
 %%%%%%%%%%%%%%
 
 \subsection{Proof of \eqref{eq:grandesvaleurs1}.}

Let $(i^{(n)})_{i\geq 1}$ be a sequence of integers such that $1\leq i^{(n)}\leq n$ and $(c_i^{(n)})_{i=i^{(n)}+1,\dots, n, n\geq 1}$ be a doubly-indexed sequence of real parameters with $c_i^{(n)}>1$, to be fixed later on.

\subsection*{Upper bound:}
From \eqref{eq:Sn=Snm} we have
\begin{align*}
    \mathbb{P}[S_n\leq t_n]
    & \leq \mathbb{P}\left[\frac{T_1'}{r^{2(U_1+\dots+U_n)}}\leq t_n, S_2^n\leq t_n\right]\\
    & \leq \mathbb{P}\left[\frac{T_1'}{r^{2(U_1+\dots+U_n)}}\leq t_n, S_2^n\leq t_n, 2(U_1+\dots+U_n)\leq c_n^{(n)}.n\right]\\
    &\qquad  +\mathbb{P}[2(U_1+\dots+U_n)>c_n^{(n)}.n]\\
    & \leq \mathbb{P}[T_1'\leq t_n  r^{c_n^{(n)}.n}, S_2^n\leq t_n]+\mathbb{P}[2(U_1+\dots+U_n)>c_n^{(n)}.n]\\
    &  \leq \mathbb{P}[T\leq t_n r^{c_n^{(n)}.n}] \times \mathbb{P}[ S_{n-1}\leq t_n]+\mathbb{P}[2(U_1+\dots+U_n)>c_n^{(n)}.n].
\end{align*}

Iterating the estimate,
$$\mathbb{P}[ S_{n\!-\!1}\leq t_n] \leq \mathbb{P}[T\leq t_n r^{c_{n\!-\!1}^{(n)}.(n\!-\!1)}] \times \mathbb{P}[ S_{n\!-\!2}\leq t_n]+\mathbb{P}[2(U_1+\dots+U_{n\!-\!1})>c_{n\!-\!1}^{(n)}.(n\!-\!1)],$$
and so on down to $i^{(n)}+1$, we obtain
\begin{equation}\label{MaxSn}
    \begin{aligned}
        \mathbb{P}[S_n\leq t_n]
        & \leq \left(\prod_{i=i^{(n)}+1}^n \mathbb{P}[T\leq t_n r^{c_i^{(n)}.i}] \right) \times \mathbb{P}[S_{i^{(n)}} \leq t_n] \\
        & + \sum_{i=i^{(n)}+1}^n \left( \prod_{j=i+1}^n \mathbb{P}[T\leq t_n r^{c_j^{(n)}.j}]\right)\times \mathbb{P}[2(U_1\!+\!\dots\!+\!U_i)>c_i^{(n)}.i].
    \end{aligned}
\end{equation}
%%%%%%

\subsection*{Choice of $i^{(n)}$ and the $c_i^{(n)}$}

Let $i^{(n)}=\lfloor \ln_2 n \rfloor$ and for $i^{(n)}+1 \leq i \leq n$,
\begin{equation}\label{eq:cni}
 c_i^{(n)}=1+\sqrt{\frac{8}{i}(\ln i+\ln_2 n)}.
\end{equation}

By (\ref{eq:tn}),
we have for $i^{(n)}+1\leq i \leq n$ and large $n$,
\begin{align} \nn
    \ln \mathbb{P}[ T\leq t_n r^{c_i^{(n)}.i}]
    & \leq \ln \mathbb{P}[ T\leq r^{c_i^{(n)}.i}] \\ \nn
    & \leq -\mathbb{P}[ T\geq r^{c_i^{(n)}.i}]  \\ \nn
    & \leq  -\frac{1}{c_i^{(n)}.i} + \eps_{n,i,1} \qquad  \qquad (by\;  \eqref{Queueminp})  \\ \nn
    & \leq  -\frac{1}{i} + \eps_{n,i,2} \qquad  \qquad (by\;  \eqref{eq:cni}) ,
\end{align}

with error terms
$$ 
\eps_{n,i,1}=\frac{\ln_2\left(c_i^{(n)}.i \ln r\right) +C}{\left(c_i^{(n)}.i\right)^2 \ln r}, \quad  \eps_{n,i,2}= \eps_{n,i,1}+\sqrt{\frac{8}{i^3}(\ln i+\ln_2 n)}\;.
$$

One can check that  $\sup_n \sum_{i=i^{(n)}+1}^n \eps_{n,i,2}< \8$, so for some positive constant $D$, for $n$ large and $i^{(n)}\leq i \leq n$,
\begin{align}\nn
    \prod_{j=i+1}^n \mathbb{P}[ T\leq t_n r^{c_j^{(n)}.j}]
    & \leq \exp\left(-\sum_{j=i+1}^n \frac{1}{j} + \sum_{j=i+1}^n \eps_{n,j,2} \right)\\ \nn
    & \leq D \exp\left( -\ln\left(\frac{n}{i}\right) \right)\\
    & \leq  D \frac{i}{n}. \label{eq:D'}
\end{align}

Combining this with  \eqref{eq:hoeffdingg}, we get for $n$ large and $i^{(n)}+1\leq i\leq n$,
\begin{align*}
    \left( \prod_{j=i+1}^n \mathbb{P}[T\leq t_n r^{c_j^{(n)}.j}]\right)\times \mathbb{P}[2(U_1\!+\!\dots\!+\!U_i)>c_i^{(n)}.i]
    & \leq D \frac{i}{n} \exp(-4(\ln i + \ln_2 n))\\
    &= \frac{D}{i^3 n (\ln n)^4}.
\end{align*}

Thus, the series $\sum a_n$, with
 $$a_n=\sum_{i=i^{(n)}+1}^n \left( \prod_{j=i+1}^n \mathbb{P}[T\leq t_n r^{c_j^{(n)}.j}]\right)\times \mathbb{P}[2(U_1\!+\!\dots\!+\!U_i)>c_i^{(n)}.i],$$
is convergent.

\subsubsection*{Choice of $t_n$}

To conclude, we need to take care of the first term in the right-hand side of \eqref{MaxSn}.
Recall $t_n$ from \eqref{eq:tn} (we will assume $n$ large so that $\ln_2 n \geq \beta$), and fix %$i_1$ be such that $\ln_2 i_1 \geq \beta$.
an integer $i_1 \geq 1$.
For $1\leq i\leq i_1$, applying \eqref{eq:LtailT0} we get as $n \to \8$, for any $\epsilon \in (0,r-1)$,
\begin{align*}
    \mathbb{P}[ T\leq t_n r^{2i}]
 %   & =\mathbb{P}[ T\leq t_n r^{2i}]\\
    & \leq  \exp\left(-\frac{(r-1-\epsilon)^2}{2\beta r^{2i}}\ln_2 n\right),
\end{align*}
and then, for n large,
\begin{align*}
    \mathbb{P}[S_{i^{(n)}} \leq t_n] 
    & \leq \mathbb{P}[T_i'\leq t_n r^{2i}, i=1,\dots,i_1]\\ 
    & =  \prod_{i=1}^{i_1} \mathbb{P}[ T\leq t_n r^{2i}] \\
    & \leq  \exp\left(-\sum_{i=1}^{i_1} \frac{(r-1-\epsilon)^2}{2\beta r^{2i}}\ln_2 n\right)\\
    & \leq  \exp\left(- \frac{(r-1-\epsilon)^2}{2\beta} \frac{1}{r^2}\frac{1 -\left( \frac{1}{r^2}\right)^{i_1}}{1-\frac{1}{r^2}}\ln_2 n\right)\\
    & \leq (\ln n)^{-\frac{(r-1-\epsilon)^2}{2\beta(r^2-1)} \left(1 -\left(\frac{1}{r^2}\right)^{i_1}\right)}.
\end{align*}
Using \eqref{eq:D'} we will bound 
$$\left(\prod_{i=i^{(n)}+1}^n \mathbb{P}[T\leq t_n r^{c_i^{(n)}.i}] \right) \times \mathbb{P}[S_{i^{(n)}} \leq t_n] 
\leq D \frac{i^{(n)}}{n} (\ln n)^{-\frac{(r-1-\epsilon)^2}{2\beta(r^2-1)} (1 -\left(\frac{1}{r^2}\right)^{i_1})},$$
where $i^{(n)}= \lfloor \ln_2 n \rfloor$.
As soon as $\beta < \frac{(r-1)}{2(r+1)}$, there exists some integer $i_1$ and some $\epsilon \in (0,r-1)$ such that
$$\frac{(r-1-\epsilon)^2}{2\beta (r^2-1)} \left(1 -\left(\frac{1}{r^2}\right)^{i_1}\right)>1,$$ and combining \eqref{MaxSn}
with $\sum_n a_n < \8$, we 
obtain $\sum \IP(S_n \leq  t_n) <\8$. i.e., 
$$\sum_{n\geq 1} \mathbb{P}[T_n\leq A_n^2 t_n]<\infty.$$

\subsubsection*{Conclusion}

Let $\beta   < \frac{(r-1)}{2(r+1)}$.
It follows from Borel-Cantelli's lemma that a.s., eventually 
$$T_n\geq \frac{\beta A_n^2}{ \ln_2 n}.$$
Now, for $T_n\leq t \leq T_{n+1}$, if $n$ is large enough, 
$$M_t\leq M_{T_{n+1}}=A_{n+1}\leq  r A_n \leq r \sqrt{\beta^{-1} T_n \ln_2 n}  \leq r  \sqrt{\beta^{-1} t \ln_2 n},$$
and since we have
 $T_n \geq \frac{\beta A_n^2}{\ln_2 n} \geq r_-^{\frac{n}{2}}$ for $n$ large enough, we have $t\geq  r_-^{\frac{n}{2}}$, and $n \leq \frac{2 \ln t}{\ln r_-}$. Finally,
$$M_t\leq r \sqrt{\beta^{-1} t \ln_2 \left(\frac{2 \ln t}{\ln r_-}\right)}.$$
Hence, we have proved \eqref{eq:grandesvaleurs1} with any
 $K>r\sqrt{\frac{2(r+1)}{(r-1)}}$. 
 % Optimizing on $r>1$, we get the result for all $K>\sqrt{11+5\sqrt{5}}\approx 4.7096$. 

 \subsection{Proof of \eqref{eq:grandesvaleurs2}} 
 
 We start by proving that it suffices to show divergence of the series introduced above \eqref{eq:tn}:
 %%%%%
 \begin{lem} \label{lem:beta0}
 Let $\beta_0 = \inf\{ \beta>0:  \sum_n \IP( S_n \leq \frac{\beta}{\ln_2 n})= \8\}$. Then
 $$
 \liminf_n S_n \ln_2 n = \beta_0 \quad a.s.
 $$
 \end{lem}
%%%
\begin{proof} For all $\beta < \beta_0$, we have $ \sum_n \IP( S_n \leq \frac{\beta}{\ln_2 n})< \8$ and the first Borel-Cantelli's lemma shows that $ \liminf_n S_n \ln_2 n \geq \beta_0$.
To prove the reverse inequality we proceed by steps:
\medskip

$\bullet$ {\it First step}: 
For any non-increasing sequence $(t_n)_n$,
\begin{equation}\nn
 \sum_{n\geq 1} \mathbb{P}[S_n\leq t_n] = \infty \implies
 \IP(  S_n \leq t_n \, i.o.) \geq \frac14 .
\end{equation}
Indeed, for $1 \leq n \leq m$, 
    \begin{align*}
        \mathbb{P}[S_n \leq t_n , S_m \leq t_m]
        & \leq \mathbb{P}[S_n \leq t_n , S_{n+1}^m \leq t_m]\\
        & = \mathbb{P}[S_n\leq t_n] \times  \mathbb{P}[S_{n+1}^m\leq t_m]\\
        & = \mathbb{P}[S_n\leq t_n]  \times \mathbb{P}[S_{m-n}\leq t_m]\\
        & \leq \mathbb{P}[S_n\leq t_n]  \times  \mathbb{P}[S_{m-n}\leq t_{m-n}],
    \end{align*}
  since $t_m \leq t_{m-n}$. Now, for $k \geq 1$,
    \begin{align*}
        \sum_{1\leq n<m\leq k} \mathbb{P}[S_n \leq t_n , S_m \leq t_m]
        & \leq \sum_{1\leq n<m\leq k} \mathbb{P}[S_n \leq t_n] \times \mathbb{P}[S_{m-n}\leq t_{m-n}]\\
        & \leq \sum_{1\leq n,m\leq k} \mathbb{P}[S_n \leq t_n]  \times \mathbb{P}[S_m\leq t_m].
    \end{align*}
For all $k$ large enough we have $\sum_{n=1}^k \mathbb{P}[S_n \leq t_n] \geq 2$, and then for all $1\leq n\leq k$, 
    $$\sum_{1\leq m\leq k, m\neq n} \mathbb{P}[S_m\leq t_m]\geq 2-\mathbb{P}[S_n\leq t_n]\geq \mathbb{P}[S_n\leq t_n].$$
Therefore,
    \begin{align*}
        \sum_{1\leq n,m\leq k} \mathbb{P}[S_n \leq t_n]  \times  \mathbb{P}[S_m\leq t_m]
        & \leq 2 \sum_{1\leq n, m\leq k, n\neq m}  \mathbb{P}[S_n \leq t_n]  \times  \mathbb{P}[S_m\leq t_m]\\
        & = 4 \sum_{1\leq n< m\leq k}  \mathbb{P}[S_n \leq t_n] \times  \mathbb{P}[S_m\leq t_m]\\
    \end{align*}
Kochen-Stone's theorem \cite{KS64} -- a variant of Borel-Cantelli's lemma -- yields 
    $$\mathbb{P}[S_n\leq t_n \, i.o.]\geq \limsup_{k\geq 1} \frac{\sum_{1\leq n< m\leq k}  \mathbb{P}[S_n \leq t_n]  \times \mathbb{P}[S_m\leq t_m]}{\sum_{1\leq n< m\leq k}  \mathbb{P}[S_n \leq t_n, S_m\leq t_m]}\geq \frac14,$$
which concludes this step.
\medskip

$\bullet$ {\it Second step}: Let's introduce the $\sigma$-fields
\begin{equation}\nn
{\mathcal A}_k= \sigma( (A_n',T_n'); n \geq k), k = 1,2\ldots , \qquad {\mathcal T}= \bigcap_{k\geq 1} {\mathcal A}_k . 
\end{equation}
By Kolmogorov 0--1 law and independence of the sequence $((A_n',T_n'); n \geq 1)$, every element $A$ of the tail field ${\mathcal T}$ has $\IP(A) \in \{0,1\}$.
Fix $\beta \geq 0$ and introduce the events
\begin{equation}\nn
E= \{ \liminf_n S_n \ln_2 n \leq \beta\}, \quad E_k= \{ \liminf_n S_{k+1}^{n+k} \ln_2 n \leq \beta\},
\end{equation}
and
\begin{equation}\nn
\Omega_0=\{ \lim_{n \to \8} \frac{\ln_2 n}{r^{2(U_1+\ldots+U_n)}}=0\}.
\end{equation}
Note that $E=E_0$ and that $\IP(\Omega_0)=1$. Since, by definition, 
$$
S_{k+1}^{n+k+1}=   \frac{T_{k+1}'}{r^{2(U_{k+1}+\ldots+U_{n+k+1})}}       +   S_{k+2}^{n+k+1} \;,
$$
we see that the two sets $E_k$ and $E_{k+1}$ coincide on $\Omega_0$, for all $k \geq 0$. Denoting the common intersection by
\begin{equation}\nn
\widehat E = E \bigcap \Omega_0 = E_k \bigcap \Omega_0 \;,
\end{equation}
we see that $\widehat E$ belongs to ${\mathcal T}$
and then has probability equal to 0 or 1. The similar 0--1 law holds for $E$ which is equal to  $\widehat E$ up to a negligible set.
\medskip

$\bullet$ {\it Final step}: 
For any $\beta >\beta_0$, the series  $\sum_n \IP( S_n \leq t_n)$ with $t_n= \beta/\ln_2 n$ is diverging. By the first step, the probability $\mathbb{P}[S_n\leq t_n \, i.o.] \geq 1/4$, and by the second one is equal to 1. Thus 
$\liminf_n S_n \ln_2 n \leq \beta$ a.s., for all such $\beta$'s. The lemma is proved.
\end{proof}
%%%%%%%%%%%
\begin{rem}\label{rem:6.2}
We have followed the approach of the renewal structure to get the 0--1 law, with the advantage to keep the paper self-contained. 
A tempting alternative would be to show that the tail $\sigma$-field of $R$ is trivial; we mention
 the illuminating  survey \cite{Ro88} on  the tail $\sigma$-field of a diffusion. 
\end{rem}
%%%%%%%
Anticipating on the proof of \eqref{eq:grandesvaleurs2}
we now give a short proof of Theorem \ref{cor:OC}.
\begin{proof} %Since the diffusion coefficient of $R$ is equal to 1, it 
It is not difficult to check the criteria of \cite{FO78} or \cite{Ros79} for triviality of  the tail $\sigma$-field of 
one-dimensional diffusion (see Theorem 3 in \cite{Ro88}).
%
%%%%%% [Details ???]
%
%[VOIR cahier bleu 21/01/21??], 
%\\
Then, 
$K^*=\limsup_{t \to \8} \frac{M(t)}{\sqrt{t \ln_3 t}}$ is a.s. constant. Then, \eqref{eq:grandesvaleurs1} and \eqref{eq:grandesvaleurs2} show that $K^*$ is positive and finite.
\end{proof}

%%%%%%%%%%%%%%%
To continue the proof of \eqref{eq:grandesvaleurs2} we need an intermediate result.
%%%%%%
\begin{lem} \label{lem:mino}
For all $\alpha_0 >0$, there exists $\beta>0$ such that, for all $n$ large enough,
    $$\mathbb{P}\big[S_{\lfloor \alpha_0 \ln_2 n \rfloor}\leq \frac{\beta}{\ln_2 n}\big]\geq \frac{1}{\ln n}.$$
\end{lem}
\begin{proof}
Clearly, it suffices to prove that for $v>0$, there exists $u>0$ such that, for all large $n$ we have, 
 \begin{equation} \label{eq:lem:mino}
   \mathbb{P}[S_n\leq \frac{u}{n}]\geq \frac{1}{e^{vn}} .
\end{equation}
Indeed, substituting $v, n$ in \eqref{eq:lem:mino} by $\alpha_0^{-1}, \lfloor  \alpha_0 \ln_2 n \rfloor$ shows that any $\beta>u/\alpha_0$ fulfills the statement of the lemma.

To show \eqref{eq:lem:mino}, we fix  some $b \in (0,1)$ ($b$ will be chosen small later on), and we note that:
 \begin{quotation}
 $U_i\geq b$ and $T_i'\leq \frac{u}{n} (r^b-1)r^{b(n-i+1)}$ for all $i=1,\ldots, n$ 
\end{quotation}   
imply that   
    $$S_n=\sum_{i=1}^n \frac{T_i'}{r^{2(U_{i}+\dots+U_n)}}\leq \sum_{i=1}^n \frac{\frac{u}{n} (r^b-1) r^{b(n-i+1)}}{r^{2b(n-i+1)}}\leq \frac{u}{n}.$$
  Then,
    \begin{align}
        \mathbb{P}[S_n\leq \frac{u}{n}] \nn
        & \geq \prod_{i=1}^n \mathbb{P}[U_i\geq b, T_i'\leq  \frac{u}{n} (r^b-1)r^{b(n-i+1)}]\\
        & = (1-b)^n \prod_{i=1}^n \mathbb{P}[ T_i'\leq  \frac{u}{n} (r^b-1)r^{b(n-i+1)}|U_i\geq b] \nn \\
               & = (1-b)^n \prod_{i=1}^n \mathbb{P}[ T\leq  \frac{u}{n} (r^b-1)r^{bi}|U\geq b] \label{eq:minprod}
    \end{align}
  By Proposition \ref{prop:LtailT}, we can find  $t_0>0$ and $\rho>0$ such that, for  $t\leq t_0$,
    $$\mathbb{P}[T \leq t | U\geq b] \geq \exp(-\frac{\rho}{t}).$$
    Now, we fix some $t_1>t_0$, we will bound the factors in \eqref{eq:minprod} as follows:\\
For  $\frac{\ln(t_1 \frac{n}{u(r^b-1)})}{b \ln r}\leq i\leq n$ :
    $$\mathbb{P}[ T\leq  \frac{u}{n} (r^b-1)r^{bi}|U\geq b]\geq \mathbb{P}[ T\leq  t_1|U\geq b],$$
for  $\frac{\ln(t_0 \frac{n}{u(r^b-1)})}{b \ln r}\leq i \leq \frac{\ln(t_1 \frac{n}{u(r^b-1)})}{b \ln r}$ :
    $$\mathbb{P}[ T\leq  \frac{u}{n} (r^b-1)r^{bi}|U\geq b]\geq \mathbb{P}[ T\leq  t_0|U\geq b],$$
and for $1\leq i \leq \frac{\ln(t_0 \frac{n}{u(r^b-1)})}{b \ln r}$ :
    $$\mathbb{P}[ T\leq  \frac{u}{n} (r^b-1)r^{bi}|U\geq b]\geq \exp\left(-\rho \frac{n}{u(r^b-1)}\frac{1}{r^{bi}}\right).$$
With this choice, the estimate  \eqref{eq:minprod}  becomes
    \begin{align*}
        \mathbb{P}[S_n\leq \frac{u}{n}]
%        & \geq \prod_{i=1}^n \mathbb{P}[U_i\geq b, T_i'\leq  \frac{u}{n} (r^b-1)r^{bi}]\\
        & \geq (1-b)^n \times \mathbb{P}[ T\leq  t_1|U\geq b]^n  \times \mathbb{P}[ T\leq  t_0|U\geq b]^{\frac{\ln(\frac{t_1}{t_0})}{b \ln r}+1}\\
        & \times \prod_{i=1}^{\lfloor \frac{\ln(t_0 \frac{n}{u(r^b-1)})}{b \ln r}\rfloor}  \exp\left(-\rho \frac{n}{u(r^b-1)}\frac{1}{r^{bi}}\right)\\
        & \geq (1-b)^n \times \mathbb{P}[ T\leq  t_1|U\geq b]^n  \times \mathbb{P}[ T\leq  t_0|U\geq b]^{\frac{\ln(\frac{t_1}{t_0})}{b \ln r}+1}\\
        & \times \exp\left(-\rho \frac{n}{u(r^b-1)^2}\right)
    \end{align*}
 From this we derive the claim \eqref{eq:lem:mino} by taking $b$ small, $u$ and $t_1$ large. This ends the proof of the lemma.
\end{proof}
%%%%%%%%%%%%%%%%%
\begin{proof} {\it Theorem \ref{th:grandesvaleurs}, claim \eqref{eq:grandesvaleurs2}.}
Similarly to the proof of \eqref{eq:grandesvaleurs1}, we let  $t_n=\frac{\beta}{\ln_2 n} \wedge 1$, $(i^{(n)})_{n\geq1}$ be a sequence of integers, and $(b_i^{(n)})_{i=i^{(n)}+1,\dots n, n\geq 1}$ be a doubly-indexed sequence  with $0 < b_i^{(n)} < 1,$ given by 
$$b_i^{(n)}=1-\sqrt{\frac{8}{i}(\ln i+\ln_2 n)}\;, \quad {\rm for} \; i^{(n)}+1 \leq i \leq n, \;  i^{(n)}=\lfloor \alpha_0 \ln_2 n \rfloor$$ %and $ i^{(n)}=\lfloor \alpha_0 \ln_2 n \rfloor$ 
with $\alpha_0$ large (take $\alpha_0>8$ so that $b_i^{(n)}>0$ for $n$ large). 

This time, we need an extra doubly-indexed, positive  sequence $(s_i^{(n)})_{i=i^{(n)}+1,\dots,n, n\geq 1}$ such that for $n$ large
$$\sum_{i=i^{(n)}+1}^n s_i^{(n)} \leq t_n\;.$$
(Note that this implies $s_i^{(n)}\leq 1$.)
Similarly, using \eqref{eq:Sn=Snm} we estimate
\begin{align*}
    \mathbb{P}[S_n\leq t_n]
    & \geq \mathbb{P}\left[\frac{T_1'}{r^{2(U_1+\dots+U_n)}}\leq s_n^{(n)}, S_2^n\leq t_n-s_n^{(n)}\right]\\
    & \geq \mathbb{P}\left[\frac{T_1'}{r^{2(U_1+\dots+U_n)}}\leq s_n^{(n)}, S_2^n\leq t_n-s_n^{(n)}, 2(U_1+\dots+U_n)\geq b_n^{(n)}.n\right]\\
    & \geq \mathbb{P}\left[T_1'\leq s_n^{(n)} r^{b_n^{(n)}.n}, S_2^n\leq t_n-s_n^{(n)}, 2(U_1+\dots+U_n)\geq b_n^{(n)}.n\right]\\
    & \geq \mathbb{P}\left[T_1'\leq s_n^{(n)} r^{b_n^{(n)}.n}, S_2^n\leq t_n-s_n^{(n)}\right] - \mathbb{P}[2(U_1+\dots+U_n)<b_n^{(n)}.n]\\
    &  \geq \mathbb{P}\left[T\leq s_n^{(n)} r^{b_n^{(n)}.n}\right] \times \mathbb{P}\left[ S_{n-1}\leq t_n-s_n^{(n)}\right] - \mathbb{P}[2(U_1+\dots+U_n)<b_n^{(n)}.n].
\end{align*}
%%%%%%%
We iterate the procedure,
\begin{align*}
    \mathbb{P}[S_{n-1}\leq t_n-s_n^{(n)}] 
    & \geq \mathbb{P}\left[T\leq s_{n-1}^{(n)}.r^{b_{n-1}^{(n)}.(n-1)}\right] \times \mathbb{P}\left[ S_{n-2}\leq t_n-s_n^{(n)}-s_{n-1}^{(n)}\right]  \\
    & - \mathbb{P}[2(U_1+\dots+U_{n-1})<b_{n-1}^{(n)}.(n-1)],
\nn \end{align*}
and so on down to $i^{(n)}$. We obtain
%%%%%%
%
%
%
%
%
%%%%%%%%%%%%%%%%%%%
%\begin{equation} 
    \begin{align}  \nn
        \mathbb{P}[S_n\leq t_n]
        & \geq \left(\prod_{i=i^{(n)}+1}^n \mathbb{P}\left[T\leq s_i^{(n)} r^{b_i^{(n)}.i}\right] \right) \times \mathbb{P}\left[ S_{i^{(n)}}\leq t_n-\sum_{i=i^{(n)}+1}^n s_i^{(n)}\right]   \\ \label{eqLiterstop} 
        & - \sum_{i=i^{(n)}\!+\!1}^n \!\left( \prod_{j=i+1}^n \mathbb{P}\left[T\leq s_j^{(n)} r^{b_j^{(n)}.j}\right]\right)\!\times \mathbb{P}[2(U_1\!+\!\dots\!+\!U_i)\!<\!b_i^{(n)}.i].
    \end{align}
%\end{equation}

Using $s_i^{(n)}\leq 1$ and $b_i^{(n)}<1$, we have, for $n$ large and $i^{(n)}+1\leq i\leq n$ :
    \begin{align*} 
        \prod_{j=i+1}^n \mathbb{P}[T\leq s_j^{(n)} r^{b_j^{(n)}.j}]
        & \leq \prod_{j=i+1}^n \mathbb{P}[T\leq r^j]\\
        & \leq \exp\left(- \sum_{j=i+1}^n \mathbb{P}[T\geq r^j]\right)\\
        & \leq \exp\left(- \sum_{j=i+1}^n \left(\frac{1}{j} -\frac{\ln_2(j \ln r)+C}{j^2 \ln r}\right)\right) \qquad  \qquad (by\;  \eqref{Queueminp})  \\
        & \leq D' \frac{i}{n},
    \end{align*}
for some positive constant $D'$.

As we did for the series $\sum_n a_n$, cf. below \eqref{eq:D'} except for using \eqref{eq:hoeffdingl} instead of \eqref{eq:hoeffdingg}, we easily see that the series $\sum_n a_n'$, with
$$a'_n=\sum_{i=i^{(n)}+1}^n \!\left( \prod_{j=i+1}^n \mathbb{P}[T\leq s_j^{(n)} r^{b_j^{(n)}.j}]\right)\times \mathbb{P}[2(U_1\!+\!\dots\!+\!U_i)<b_i^{(n)}.i],$$
is finite.  Now, we choose 
$$s_i^{(n)}=\frac{1}{i^3}\;,$$
and we start to bound from below  the product
\begin{align*}
    \prod_{i=i^{(n)}+1}^n \mathbb{P}[T\leq s_i^{(n)} r^{b_i^{(n)}.i}] 
    & = \exp\left(\sum_{i=i^{(n)}+1}^n \ln (1 - \mathbb{P}[ T\geq s_i^{(n)} r^{b_i^{(n)}i}])\right)
\end{align*}
Observe that, by taking $\alpha_0>16$, we have $b_i^{(n)} \in (1/2, 1)$ for all  large $n$ and $i \in [i^{(n)}+1,n]$, and also that
 \begin{equation} \label{eq:crucial} 
% \inf\{ s_i^{(n)} r^{b_i^{(n)}.i} ;  i^{(n)}\leq i\leq n \} \to \infty \quad {\rm as\ } n \to \8.
 \inf\{ s_i^{(n)} r^{b_i^{(n)}.i} ;  i^{(n)}\leq i\leq n \} \geq r^{\frac {\alpha_0}{2} \ln_2 n} \qquad {\rm for \ large\ } n ,
\end{equation}
which tends to $\8$ as $n \to \8$.
%%%%%%%%%%%%%%%%%%%%%   
For $i^{(n)}+1\leq i\leq n$ and $n$ large,  in view of \eqref{eq:crucial} we have (using $-\ln(1-u)\leq u+u^2$  for small $u>0$ and $\frac{1}{1-u}\leq 1+2u$ for $0<u<\frac12$)
\begin{align*}
    -\ln (1 - \mathbb{P}[ T\geq s_i^{(n)} r^{b_i^{(n)}.i}]) 
    &\leq \mathbb{P}[ T\geq s_i^{(n)} r^{b_i^{(n)}.i}]  + \eps'_{n,i,1}\\
    & \leq \frac{\ln r}{\ln\left(s_i^{(n)}  r^{b_i^{(n)}.i}\right)}    + \eps'_{n,i,2}  \qquad ({\rm by\ } \eqref{Queueminp})\\
    & = \frac{1}{b_i^{(n)}.i +\frac{\ln s_i^{(n)}}{\ln r}} + \eps'_{n,i,2}   \\
    & \leq \frac{1}{b_i^{(n)}.i}  + \eps'_{n,i,3}  \\
    & \leq \frac{1}{i} + \eps'_{n,i,4} \;,
\end{align*}
with error terms 
$$ 
\eps'_{n,i,1}=\mathbb{P}\left[ T\geq s_i^{(n)} r^{b_i^{(n)}.i}\right]^2, \quad  \eps'_{n,i,2}= \eps'_{n,i,1}+
 \frac1{\ln r}\times \frac{\ln_3 \left(   s_i^{(n)}  r^{b_i^{(n)}.i} \right)    +C} {\big( b_i^{(n)}.i +\frac{\ln s_i^{(n)}}{\ln r}\big)^2},
$$
$$
\eps'_{n,i,3}=\eps'_{n,i,2} - 2 \frac{\ln s_i^{(n)}}{\left(b_i^{(n)}. i\right)^2\ln r},\quad 
\eps'_{n,i,4}= \eps'_{n,i,3}  + 2 \sqrt{\frac{8}{i^3}(\ln i+\ln_2 n)}.
$$

One can check that  $\sup_n \sum_{i=i^{(n)+1}}^n \eps'_{n,i,4}< \8$, so for some positive constant $D''$, for large $n$,
%\begin{equation} 
    \begin{align}  \nn
        \prod_{i=i^{(n)}+1}^n \mathbb{P}[T\leq s_i^{(n)} r^{b_i^{(n)}.i}] 
        & \geq \exp\left(-\sum_{i=i^{(n)}+1}^n \left(
        \frac{1}{i} + \eps'_{n,i,4} \right)\right)\\  \label{eqhalali}
        & \geq  D'' \frac{i^{(n)}}{n}. 
    \end{align}
%\end{equation}

Finally, consider  the term
$$ \mathbb{P}\left[ S_{i^{(n)}}\leq t_n-\sum_{i=i^{(n)}+1}^n s_i^{(n)}\right].$$
Note that $t_n-\sum_{i=i^{(n)}+1}^n s_i^{(n)}=\frac{\beta}{\ln_2 n} - \sum_{i=i^{(n)}+1}^n \frac{1}{i^3}\geq \frac{\beta}{\ln_2 n} -  \frac{1}{2 \:{i^{(n)}}^2}$,
which implies that for all  $\beta'<\beta$, $t_n-\sum_{i=i^{(n)}+1}^n s_i^{(n)}\geq \frac{\beta'}{\ln_2 n}$ for large $n$, and then
$$ \mathbb{P}\left[ S_{i^{(n)}}\leq t_n-\sum_{i=i^{(n)}+1}^n s_i^{(n)}\right] \geq  \mathbb{P}\left[ S_{i^{(n)}}\leq  \frac{\beta'}{\ln_2 n}\right].$$
\medskip

Now, we are ready to conclude the proof: Fix  $\alpha_0>16$, and let $\beta'$ be associated to $\alpha_0$ by Lemma  \ref{lem:mino}. Then,
$$ \mathbb{P}\left[ S_{i^{(n)}}\leq  \frac{\beta'}{\ln_2 n}\right]\geq \frac{1}{\ln n},$$
and for $t_n=(\beta/\ln_2 n) \wedge 1$ with $\beta>\beta'$, using \eqref{eqhalali},
$$\left( \prod_{i=i^{(n)}+1}^n \mathbb{P}[T\leq s_i^{(n)} r^{b_i^{(n)}.i}] \right) \times \mathbb{P}\left[ S_{i^{(n)}}\leq t_n-\sum_{i=i^{(n)}+1}^n s_i^{(n)}\right] \geq  D'' \frac{i^{(n)}}{n} \times  \frac{1}{\ln n}.$$
Using now \eqref{eqLiterstop} and $\sum_n a_n'<\8$ we obtain
$\sum_{n\geq 1} \mathbb{P}[S_n\leq t_n] = \infty.$ By Lemma \ref{lem:beta0} 
we have a.s.,
$$T_n\leq \frac{\beta A_n^2}{ \ln_2 n} \quad i.o.$$
i.e.,  $A_n\geq \sqrt{\beta^{-1}T_n  \ln_2 n}$. Since, for all large $n$, $\frac{\beta A_n^2}{\ln_2 n} \leq r_+^n$, we see that $T_n\leq r_+^n$, so $n \geq \frac{\ln T_n}{\ln r_+},$ and also
$$M_{T_n}=A_n\geq \sqrt{\beta^{-1} T_n \ln_2 \left(\frac{\ln T_n}{\ln r_+}\right)}.$$
Finally, for some (small) $K'>0$, with probability one, $M_t \geq  K'  \sqrt{t \ln_3 t}$ i.o. The proof of \eqref{eq:grandesvaleurs2} is complete.
\end{proof}

%%%%%%%%%%%%
%
%
%      valeur de $K'$
%
%
%%%%%%%%%%%%
%\begin{rem}
%Précisions sur la valeur de $K'$.
%
%Soit $\beta_*=\inf \left\{u\times v | u>0, v\in (0,1/2), \min_{b\in (0,1)}  -\ln(1-b)+\frac{1}{2ub^2}<v\right\} \approx 18.5081$ (à vérifier). Alors on peut établir (9) pour tout $K'<\sqrt{\beta_*^{-1}}\approx 0,232$.
%
%En effet, nous procédons de la manière suivante. 
%Soient $u>0$, $v\in(0,\frac12)$ et $b\in (0,1)$ tels que :
%$$ -\ln(1-b)+\frac{1}{2ub^2}<v.$$
%
%Il existe alors $r>1$ tel que :
%$$ -\ln(1-b)+\frac{(r-1)^2}{2u(r^b-1)^2}<v.$$
%
%La proposition (3.6) peut être établie pour tout $\rho>\frac{(r-1)^2}{2}$ [A VERIFIER !!!], donc en prenant un tel $\rho$ satisfaisant :
%$$ -\ln(1-b)+\frac{\rho}{u(r^b-1)^2}<v,$$
%nous obtenons la formule (34) de la preuve du lemme 6.3.
%
%Maintenant, posons $\alpha_0=v^{-1}>2$. Nous obtenons alors le lemme 6.3 pour tout $\beta > u/\alpha_0=u v$.  
%
%Ensuite, la preuve de (9) s'adapte sans difficulté avec :
%$$b_i^{(n)}=1-\sqrt{\frac{8\ln i+2 \ln_2 n }{i}}\;, \quad {\rm for} \; i^{(n)} \leq i \leq n, \;  i^{(n)}=\lfloor \alpha_0 \ln_2 n \rfloor$$
%
%Finalement, nous obtenons (9) pour tout $K'<\sqrt{\beta^{-1}}$. En optimisant sur $u$, $v$ et $b$, nous obtenons les valeurs annoncées de $K'$.
%\end{rem}
%%%%%%%%

\bigskip

{\bf Acknowledgments}: FC is partially supported by ANR SWIWS.

%%%%%%%%%%
%
%
%
%%%%%%%%%%%%%%%%%%%%%%
%   \break

\end{document}